\documentclass[a4paper]{amsart}
\usepackage{graphicx,amsmath,amsfonts,latexsym,amssymb,amsthm,mathrsfs, color,hyperref}
\usepackage[latin1]{inputenc}
\usepackage{amsthm}
\usepackage{amssymb}
\theoremstyle{plain}
\newcommand{\skp}[1]{\langle#1\rangle}
\renewcommand{\Re}{\mathop{\rm Re}}

\newtheorem{thm}{Theorem}[section]
\newtheorem{proposition}[thm]{Proposition}
\newtheorem{corollary}[thm]{Corollary}
\newtheorem{definition}[thm]{Definition}
\newtheorem{theorem}[thm]{Theorem}
\newtheorem{lemma}[thm]{Lemma}
\newtheorem{remark}[thm]{Remark}
\newtheorem*{prooft1*}{Proof of Theorem 1.1}

\DeclareMathOperator{\Div}{div} 
\newcommand{\B}{{\mathbb B}}
\newcommand{\C}{{\mathbb C}}

\newcommand{\N}{{\mathbb N}}

\newcommand{\R}{{\mathbb R}}

\newcommand{\cD}{{\mathcal D}}
\newcommand{\cE}{{\mathcal E}}

\newcommand{\cH}{{\mathcal H}}

\newcommand{\cL}{{\mathcal L}}

\newcommand{\cP}{{\mathcal P}}

\renewcommand{\gg}{\gamma}

\newcommand{\gD}{\Delta}

\newcommand{\gve}{\varepsilon}

\newcommand{\go}{\omega}

\newcommand{\gvp}{\varphi}
\newcommand{\gt}{\theta}

\newcommand{\gs}{\sigma}

\begin{document}
\title[The Porous Medium Equation on Manifolds with Conical Singularities]%
{Smoothness and Long Time Existence for Solutions of the Porous Medium Equation on Manifolds with Conical Singularities}
\author{Nikolaos Roidos and Elmar Schrohe}
\address{Institut f\"ur Analysis, Leibniz Universit\"at Hannover, Welfengarten 1, 30167 Hannover, Germany}
\email{roidos@math.uni-hannover.de, schrohe@math.uni-hannover.de}

\date{\today}
\begin{abstract} We study the porous medium equation on manifolds with conical singularities. Given strictly positive initial values, we show that the solution exists in the maximal $L^q$-regularity space for all times and is instantaneously smooth in space and time, where the maximal $L^q$-regularity is obtained in the sense of Mellin-Sobolev spaces. Moreover, we obtain precise information concerning the asymptotic behavior of the solution close to the singularity. Finally, we show the existence of generalized solutions for non-negative initial data. 
\end{abstract}
\maketitle

\section{Introduction}
The porous medium equation (PME) is the parabolic diffusion equation 
\begin{eqnarray}\label{e1}
u'(t)-\Delta(u^{m}(t))&=&0,\quad t>0,\\\label{e2}
u(0)&=&u_{0}.
\end{eqnarray}
It describes the flow of a gas in a porous medium; $u$ is the density distribution of the gas, $\Delta$ is the (negative) Laplacian, and $m>0$. For $m=1$ we recover the usual heat equation.

In this article, we consider the PME on a manifold with conical singularities, which we model by a smooth compact $(n+1)$-dimensional manifold $\mathbb{B}$ with (possibly disconnected) boundary $\partial\mathbb{B}$, $n\ge1$. $\mathbb B$ is
endowed with a Riemannian metric $g$ which, in a collar neighborhood $[0,1)\times\partial\mathbb{B}$ of the boundary, is of the degenerate form
\begin{gather}\label{g}
g= dx^2+x^2h(x).
\end{gather}
Here, $x\in[0,1)$ denotes the distance from $\partial\mathbb{B}$ and $x\mapsto h(x)$
is a smooth (up to $x=0$) family of non-degenerate Riemannian metrics on the cross-section, so that $\partial\mathbb{B}$ corresponds to the conical points. 
We speak of straight conical singularities, if $h$ is constant in $x$, else of warped conical singularities. 
On the collar $(0,1)\times\partial\mathbb{B}$, the associated Laplace-Beltrami operator $\Delta $ has the conically degenerate form:
\begin{gather}\label{delta}
\Delta=\frac{1}{x^{2}}\Big((x\partial_{x})^{2}+\big(n-1+\frac{x\partial_x\det[h(x)]}{2\det[h(x)]}\big)x\partial_{x}+\Delta_{h(x)}\Big),
\end{gather}
where $\Delta_{h(x)}$ is the Laplacian on the cross-section induced by the metric $h(x)$. 
Hence $\Delta$ has the typical features of a {\em cone differential operator}: the factor $x^{-2}$ in front and the derivative with respect to $x$ only appearing in the form $x\partial_x$. It therefore naturally acts on scales of {\em weighted Mellin-Sobolev} spaces $\mathcal{H}_{p}^{s,\gamma}(\mathbb{B})$, $s,\gamma\in \mathbb{R}$, $p\in(1,\infty)$, see Section 2 for details, and this is the framework in which we will work.

The problem \eqref{e1}-\eqref{e2} has been studied in various domains and under many aspects. We here show existence of long time solutions as well as smoothing in space and time for suitably chosen Mellin-Sobolev spaces. 
Furthermore, we obtain information on the asymptotic behavior of the solution close to the conic singularities and show its dependence on the local geometry of the tip. 
As a first step in the above direction we established in \cite{RS3} existence, uniqueness and maximal $L^q$-regularity for the short time solution on Mellin-Sobolev spaces of arbitrarily high order. This result will be one of the basic tools for our analysis. 
 
The main difficulty compared to the case of classical domains is the lack of basic machinery. 
In order to overcome this problem we rely mostly on abstract maximal regularity theory for linear and quasilinear parabolic problems. The methods we develop are general and can be applied to different problems. For instance, we do not have the usual gradient estimates for the short time solution as e.g. in the case of complete manifolds with Ricci curvature bounded from below, see \cite{LV}. Instead, we show a functional analytic analog, using {\em interpolation space estimates}, see Theorem \ref{xq}, based on uniform maximal $L^q$-regularity estimates.
 
In the same spirit, we establish a general smoothing result for the abstract quasilinear equation, see Section 3, via maximal $L^q$-regularity theory on Banach scales. We prove existence for long time by controlling uniformly all the parameters of a Banach fixed point argument due to H. Amann. 
 The main result of this article concerning the porous medium equation is the theorem, below. The set-up is the same as in \cite{RS3}. 
 
 Recall that $n+1$ is the dimension of $\mathbb B$. 
Denote by $\lambda_{1}$ the greatest nonzero eigenvalue of $\Delta_{h(0)}$; 
let 
\begin{eqnarray}\label{epsilon1}
\overline \varepsilon = -\frac {n-1}2+ 
\sqrt{\left(\frac {n-1}2\right)^2-\lambda_1}>0
\end{eqnarray}
and write ${\mathcal E}_0$ for the space of smooth functions on $\mathbb B$ which are locally constant near $\partial \mathbb B$.

\begin{theorem}\label{Th1}
Choose 
\begin{eqnarray}\label{gamma}
\gamma\in \Big(\frac{n-3}{2},\frac{n-3}{2}+\min\{\overline\varepsilon,2\}\Big).
\end{eqnarray}
Moreover, fix $1<p,q<\infty$ and $s_0\in \mathbb R$ with 
\begin{eqnarray}\label{pq1}
 \frac{n+1}p+\frac2q<1, \ \gamma > \frac{n-3}{2}+\frac2q \text{ and }\\
 s_{0}>-1+\frac{n+1}{p}+\frac{2}{q}. 
\end{eqnarray}
Then:

{\rm (a)} For any $T>0$ and any strictly positive initial value 
$$
u_{0}\in (\mathcal{H}_{p}^{s_{0}+2,\gamma+2}(\mathbb{B})\oplus{\mathcal E}_0,\mathcal{H}_{p}^{s_{0},\gamma}(\mathbb{B}))_{\frac{1}{q},q}
$$
the porous medium equation \eqref{e1}-\eqref{e2} has a unique solution 
\begin{eqnarray}\label{reg125}
u\in W^{1,q}(0,T;\mathcal{H}_{p}^{s_{0},\gamma}(\mathbb{B})) \cap L^{q}(0,T;\mathcal{H}_{p}^{s_{0}+2,\gamma+2}(\mathbb{B})\oplus{\mathcal E}_0).
\end{eqnarray}

{\rm (b)} Moreover, $u$ has the following regularity properties: 
\begin{eqnarray}\nonumber
\lefteqn{u \in C([0,T], (\mathcal{H}_{p}^{s_{0}+2,\gamma+2}(\mathbb{B})\oplus{\mathcal E}_0,\mathcal{H}_{p}^{s_{0},\gamma}(\mathbb{B}))_{\frac{1}{q},q})}\nonumber\\
&&\cap \, C([0,T];\mathcal{H}_{p}^{s_{0}+2-\frac{2}{q}-\varepsilon,\gamma+2-\frac{2}{q}-\varepsilon}(\mathbb{B})+{\mathcal E}_0)\nonumber\\
\nonumber
&&\cap\, C^{\delta}((0,T];\mathcal{H}_{p}^{s,\gamma+2-\frac{2}{q}-2\delta}(\mathbb{B})+{\mathcal E}_0)\cap C^{1+\delta}((0,T];\mathcal{H}_{p}^{s,\gamma-\frac{2}{q}-2\delta}(\mathbb{B}))\\\nonumber
&&\cap\, C^{\delta}([0,T];\mathcal{H}_{p}^{s_{0}+2-\frac{2}{q}-2\delta,\gamma+2-\frac{2}{q}-2\delta}(\mathbb{B})\oplus{\mathcal E}_0)\cap C^{1+\delta}([0,T];\mathcal{H}_{p}^{s_{0}-\frac{2}{q}-2\delta,\gamma-\frac{2}{q}-2\delta}(\mathbb{B}))\\\label{reg126}
&&\cap\, C^{k}((0,T];\mathcal{H}_{p}^{s,\gamma-2(k-1)}(\mathbb{B}))
\cap\;C((0,T];\mathcal{H}_{p}^{s,\gamma+2}(\mathbb{B})\oplus{\mathcal E}_0)
\end{eqnarray}
for all $k\in\mathbb{N}\backslash\{0\}$, $s>0$, $\varepsilon>0$ and 
\begin{eqnarray}\label{delta1}
\delta\in\Big(0,\frac{1}{2}\min\Big\{2-\frac{n+1}{p}-\frac{2}{q},\gamma-\frac{n-3}{2}-\frac{2}{q}\Big\}\Big).
\end{eqnarray}

{\rm (c)} If $c_{0}\leq u_{0}\leq c_{1}$ on $\mathbb{B}$ for suitable constants $c_{0},c_{1}>0$, then also $c_{0}\leq u\leq c_{1}$ on $[0,T]\times\mathbb{B}$. 
\end{theorem}

\begin{remark}\label{rem1.2}
\rm (a) We consider here $\Delta$ as an unbounded operator in the cone Sobolev space $\mathcal H^{s_0,\gamma}_p(\mathbb B)$. Cone differential operators in general have many closed extensions.
Here we fix the extension $\underline \Delta$ with domain 
\begin{eqnarray}\label{eq:domain}
\mathcal D(\underline\Delta)= \mathcal H^{s_0+2,\gamma+2}_p(\mathbb B)\oplus {\mathcal E}_0.
\end{eqnarray}
This choice is particularly convenient as discussed in Section \ref{PME_cone}, below. 

(b) For the interpolation spaces the following embeddings hold: For every $\varepsilon>0$
\begin{eqnarray*}
\mathcal{H}_{p}^{s_{0}+2-\frac{2}{q}+\varepsilon,\gamma+2-\frac{2}{q}+\varepsilon}(\mathbb{B})+{\mathcal E}_0\hookrightarrow (\mathcal{H}_{p}^{s_{0}+2,\gamma+2}(\mathbb{B})\oplus{\mathcal E}_0,\mathcal{H}_{p}^{s_{0},\gamma}(\mathbb{B}))_{\frac{1}{q},q}\text{ and} \\
 (\mathcal{H}_{p}^{s_{0}+2,\gamma+2}(\mathbb{B})\oplus{\mathcal E}_0,\mathcal{H}_{p}^{s_{0},\gamma}(\mathbb{B}))_{\frac{1}{q},q}\hookrightarrow \mathcal{H}_{p}^{s_{0}+2-\frac{2}{q}-\varepsilon,\gamma+2-\frac{2}{q}-\varepsilon}(\mathbb{B})+{\mathcal E}_0.
\end{eqnarray*}
The sum in the second embedding is direct, provided $\gamma+2-\frac{2}{q}-\varepsilon\ge \frac{\dim \mathbb{B}}2$.

(c) The first inclusion in \eqref{reg126} is an immediate consequence of the fact that 
\begin{eqnarray}\nonumber
\lefteqn{W^{1,q}(T_{0},T_{2};\mathcal{H}_{p}^{s_{0},\gamma}(\mathbb{B})) \cap L^{q}(T_{0},T_{2};\mathcal{H}_{p}^{s_{0}+2,\gamma+2}(\mathbb{B})\oplus{\mathcal E}_0)}\\ \label{contint}
&&\hookrightarrow C([T_{0},T_{2}];(\mathcal{H}_{p}^{s_{0}+2,\gamma+2}(\mathbb{B}) \oplus{\mathcal E}_0, \mathcal{H}_{p}^{s_{0},\gamma}(\mathbb{B}))_{\frac{1}{q},q}),
\end{eqnarray}
where $0\leq T_{0}< T_{1}\leq T_{2}\leq T_{3}<\infty$, see \cite[Theorem III.4.10.2]{Am}; note that if $T_{1}$, $T_{3}$ are fixed, then the norm of the embedding is independent of $T_{0}, T_{2}$ due to \cite[Corollary 2.3]{CL}. The second inclusion then follows from (b). We will see the other inclusions later in the proof. 

(d) Theorem \ref{Th1} shows how the geometry of the cone is reflected in the solution. In fact, the choice of $\gamma$ is limited by $\overline\varepsilon$, 
determined from $\lambda_1$. Now we conclude from (a) and (b) that, for small $\varepsilon$, the solution 
$u$ is an element of $C([0,T], \mathcal{H}_{p}^{s_{0}+2-{2}/{q}-\varepsilon,\gamma+2-{2}/{q}-\varepsilon}(\mathbb{B})\oplus {\mathcal E}_0)$. 
We can therefore decompose it $u=u_{{\mathcal E}_0}+ u_{\mathcal H}$, 
where $u_{{\mathcal E}_0} \in C([0,T], {\mathcal E}_0)$ and $u_{\mathcal H}\in 
C([0,T], \mathcal{H}_{p}^{s_{0}+2-{2}/{q}-\varepsilon,\gamma+2-{2}/{q}-\varepsilon}(\mathbb{B}))$. A standard estimate for cone Sobolev spaces, see 
\cite[Corollary 2.9]{RS2}, then implies that, for a suitable constant $c$ depending only on $\mathbb B$ and $p$, 
$$|u_{\mathcal H} (t)| \le c x^{\gamma+2-{2}/{q}-\varepsilon-\dim \mathbb B/2}
\|u_{\mathcal H}\|_{ \mathcal{H}_{p}^{s_{0}+2-\frac{2}{q}-\varepsilon,\gamma+2-\frac{2}{q}-\varepsilon}(\mathbb{B})}.$$
\end{remark}

\subsubsection*{Outline of the paper} In Section 2 we collect maximal $L^q$-regularity 
results for the short time solution of the PME on conic manifolds. Section 3 contains an 
abstract smoothing result that we will use later on. As a preparation for the proof of the long
time existence, we show in Section 4 a comparison principle and a uniform Hölder 
regularity result for the solution. We then derive an estimate for the solution in a suitable
interpolation space. This allows us to prove Theorem 1.1 in Section 5. 
We finally recall the notion 
of weak solutions and show existence of weak solutions for non-negative initial data and $m\ge1$.

\subsubsection*{Related work} In a similar spirit, Shao \cite{Shao} showed global existence of $L^{1}$-mild solutions of the PME on singular manifolds, relying on a concept of Amann. Bahuaud and Vertman in \cite{Ver} considered the normalized Yamabe flow on manifolds with incomplete edge singularities and established existence and uniqueness of long time solutions. 

\section{Preliminary Results}
\setcounter{equation}{0}

\subsection{Sectoriality and maximal $L^q$-regularity } 
Let $X_{1}\overset{d}{\hookrightarrow} X_{0}$ be a continuously and densely injected complex Banach couple.

\begin{definition}
By $\mathcal{P}(\theta)$, $0\le \theta<\pi$, we denote the class of all closed densely defined linear operators $A$ in $X_{0}$ such that 
\[
S_{\theta}=\{\lambda\in\mathbb{C}\,|\, |\arg \lambda|\leq\theta\}\cup\{0\}\subset\rho{(-A)} \quad \mbox{and} \quad (1+|\lambda|)\|(A+\lambda)^{-1}\|_{\mathcal{L}(X_{0})}\leq K, \quad \lambda\in S_{\theta},
\]
for some constant $K\geq1$ that is called {\em sectorial bound} of $A$. The elements in $\mathcal{P}(\theta)$ are called {\em (invertible) sectorial operators of angle $\theta$}.
\end{definition}

Focusing on sectorial operators of angle greater than $\frac{\pi}{2}$ we recall the following basic property for the abstract linear first order Cauchy problem. Namely, let $T>0$, $q\in(1,\infty)$, $f\in L^q(0,T;X_{0})$ and consider the equation
\begin{eqnarray}\label{app1}
u'(t)+Au(t)&=&f(t), \quad t\in(0,T),\\\label{app2}
u(0)&=&0,
\end{eqnarray}
where $-A:X_{1}\rightarrow X_{0}$ is the infinitesimal generator of an analytic semigroup on $X_{0}$. We say that $A$ has {\em maximal $L^q$-regularity} if for any $f\in L^q(0,T;X_{0})$ there exists a unique $u\in W^{1,q}(0,T;X_{0})\cap L^q(0,T;X_{1})$ solving \eqref{app1}-\eqref{app2}. In this case, $u$ also depends continuously on $f$. Furthermore, the above property is independent of $q$. 

All the Banach spaces we will consider in the sequel belong to the class UMD (unconditionality of martingale differences property), i.e. the class of Banach spaces $X_{0}$ such that the Hilbert transform is bounded in $L^{2}(\mathbb{R};X_{0})$ (see \cite[Section III.4.4]{Am}). 

\begin{definition}
Let $\{\epsilon_{k}\}_{k=1}^{\infty}$ be the sequence of the Rademacher functions and $\theta\in[0,\pi)$. An operator $A\in \mathcal{P}(\theta)$ is called {\em $R$-sectorial of angle $\theta$} if there exists a constant $R\geq 1$, called {\em $R$-sectorial bound} of $A$, such that for any choice of $\lambda_{1},\ldots,\lambda_{N}\in S_{\theta}\backslash\{0\}$ and $x_{1},\ldots,x_{N}\in X_0$, $N\in\mathbb{N}\backslash\{0\}$, we have that
\begin{eqnarray*}
\Big\|\sum_{k=1}^{N}\epsilon_{k}\lambda_{k}(A+\lambda_{k})^{-1}x_{k}\Big\|_{L^{2}(0,1;X_0)} \leq R\,\Big\|\sum_{k=1}^{N}\epsilon_{k}x_{k}\Big\|_{L^{2}(0,1;X_0)}.
\end{eqnarray*}
\end{definition} 

Then, the following classical result holds. 

\begin{theorem}[{\rm Kalton and Weis, \cite[Theorem 6.5]{KW1}}]
In a UMD Banach space any $R$-sectorial operator of angle greater than $\frac{\pi}{2}$ has maximal $L^{q}$-regularity. 
\end{theorem}

If $A\in\mathcal{P}(\theta)$, $\theta\in[0,\pi)$, is a sectorial operator and $z\in\mathbb{C}$ with $\mathrm{Re}(z)<0$, then the complex powers $A^z$ of $A$ are defined by the Dunford integral
\begin{gather}\label{ppo}
A^{z}=\frac{1}{2\pi i}\int_{\Gamma}(-\lambda)^{z}(A+\lambda)^{-1}dz\in\mathcal{L}(X_{0}),
\end{gather}
for an appropriate path $\Gamma$ (see \cite[Theorem III.4.6.5]{Am}). 
The above operators are injections, and hence the powers $A^z$ for $\mathrm{Re}(z)>0$ are defined by $A^z=(A^{-z})^{-1}$, which are in general unbounded operators. 
By using the formula \eqref{ppo} we can also define the imaginary powers $A^{it}$, $t\in\mathbb{R}$, of $A$ (see \cite[(III.4.6.21)]{Am}).

\begin{definition}
Let $A\in\mathcal{P}(\theta)$, $\theta\in[0,\pi)$. We say that $A$ has bounded imaginary powers, if there exist some $\varepsilon,\delta>0$ such that 
\[
A^{it}\in\mathcal{L}(X_0) \quad \mbox{ and } \quad \|A^{it}\|\leq \delta \quad \mbox{for all } \quad t\in[-\varepsilon,\varepsilon].
\]
In this case, there exists a $\phi\geq0$, the so-called power angle, such that $\|A^{it}\|\leq M e^{\phi|t|}$ for all $t\in\mathbb{R}$ and some $M>0$, and we write $A\in\mathcal{BIP}(\phi)$.
\end{definition} 

By a result of Dore and Venni \cite[Theorem 3.2]{DV}, in a UMD Banach space every operator in the class $\mathcal{BIP}(\phi)$ with $\phi<\frac{\pi}{2}$ has maximal $L^q$-regularity.

\subsection{The PME on conic manifolds}\label{PME_cone}

As a cone differential operator, $\Delta$ naturally acts on scales of weighted Mellin-Sobolev spaces $\cH^{s,\gg}_p(\B)$, $s,\gamma \in \R$, $p\in(1,\infty)$. Denoting by $\B^\circ$ the interior of $\B$, for $s\in \N_0$, $\cH^{s,\gg}_p(\B)$ is the space of all $u\in H^s_{p,loc}(\B^\circ)$ for which, in local coordinates near the boundary, 
$$
x^{\frac{n+1}2-\gamma}(x\partial_x)^j\partial_y^\alpha u(x,y)\in L^p\left(\sqrt{\det[h(x)]}{\frac{dx}{x}dy}\right).
$$
For other values of $s$, the space $\cH^{s,\gamma}_p(\B)$ can be obtained by interpolation and duality (see \cite[Definition 3.1]{RS3} for details). Moreover, since the usual Sobolev spaces are UMD, by \cite[Theorem III.4.5.2]{Am}, the Mellin-Sobolev spaces are also UMD. 

\begin{lemma}\label{abr}
\label{c0} Let $1<p,q<\infty$ and $s>\frac{n+1}{p}$. Then:

{\rm (a)} A function $u$ in $\cH^{s,\mu}_p(\B)$, $\mu\in\mathbb{R}$, is continuous on $\B^\circ$, and, in local coordinates $(x,y)\in(0,1)\times\partial\B$ near the boundary, 
\begin{eqnarray*}
|u(x,y)|\le c x^{\mu-\frac{n+1}{2}} \|u\|_{\cH^{s,\mu}_p(\B)}
\end{eqnarray*}
for a constant $c>0$, depending only on $\B$ and $p$. Moreover, for $v\in \cH^{s,\frac{n+1}{2}}_p(\B)$ and $\gamma\in\mathbb{R}$,
$$
\|uv\|_{\cH^{s,\gg}_p(\B)}\le C \|u\|_{\cH^{s,\gg}_p(\B)}\|v\|_{\cH^{s,\frac{n+1}{2}}_p(\B)}
$$ 
for suitable $C>0$. In particular, up to the choice of an equivalent norm, $\cH^{s,\gg}_p(\B)$ is a Banach algebra whenever $\gg\ge \frac{n+1}{2}$. 

{\rm (b)} Multiplication by an element in $\cH^{\gs+\frac{n+1}{q},\frac{n+1}{2}}_q(\B)$, $\gs>0$, defines a bounded map on $\cH^{\rho,\mu}_p(\B)$, $\mu\in\R$, for each $\rho\in(-\gs,\gs)$. 

{\rm (c)} 
If $\gg> \frac{n+1}{2}$and $v\in \mathcal{H}_{p}^{s,\gamma}(\mathbb{B})\oplus{\mathcal E}_0$ is pointwise invertible, then $v^{-1}\in \mathcal{H}_{p}^{s,\gamma}(\mathbb{B})\oplus{\mathcal E}_0$, i.e. 
$\mathcal{H}_{p}^{s,\gamma}(\mathbb{B})\oplus{\mathcal E}_0$ is spectrally invariant in $C(\B)$ and therefore closed under holomorphic functional calculus. 

In this case, if $U$ is a bounded open subset of $\mathcal{H}_{p}^{s,\gamma}(\mathbb{B})\oplus{\mathcal E}_0$ consisting of functions $v$ such that 
$\Re(v)\geq\alpha>0$ for some fixed $\alpha$, then the subset $\{v^{-1}\, |\, v\in U\}$ of $\mathcal{H}_{p}^{s,\gamma}(\mathbb{B})\oplus{\mathcal E}_0$ is also bounded and for any $\nu\in\R$
\begin{gather}\label{poly}
v_{1}^{\nu}-v_{2}^{\nu}=(v_{2}-v_{1})\frac{1}{2\pi i}\int_{\Gamma}(-\lambda)^{\nu}(v_1+\lambda)^{-1}(v_2+\lambda)^{-1}d\lambda \in \mathcal{H}_{p}^{s,\gamma}(\mathbb{B})\oplus{\mathcal E}_0, \quad v_{1}, v_{2}\in U,
\end{gather}
where $\Gamma$ is a fixed finite path around $\cup_{v\in U}\mathrm{Ran}(-v)$ in $\{ \lambda\in\C\, |\, \Re(\lambda)<0\}$.
\end{lemma}
\begin{proof}
This is \cite[Lemma 3.2]{RS3}, \cite[Corollary 3.3]{RS3}, \cite[Lemma 6.2]{RS3}, \cite[Lemma 6.3]{RS3} and \cite[(6.19)]{RS3}.
\end{proof}

We will consider first the Laplacian associated with a straight conical metric as an unbounded operator in $\cH^{s,\gg}_p(\B)$. 
In this case, the metric $h$ in \eqref{g} is constant in $x$ and we write $\Delta_{h(0)} = \Delta_\partial$ for the boundary Laplacian.
It is well known, see e.g. \cite{RS3}, that the domain of the minimal extension of $\Delta$, i.e. the closure of $\Delta$ 
considered as an operator on $C^\infty_c(\B^\circ)$, differs from that of the maximal extension, 
which consists of all $u\in \cH^{s,\gg}_p(\B)$ such that $\Delta u \in \cH^{s,\gg}_p(\B)$, by a finite dimensional space $\cE$. 
In order to understand this space, recall that the conormal symbol	
$\sigma_M(\Delta)$ of $\Delta$ is the operator-valued function 
\begin{eqnarray}\label{eq:conormal}
\sigma_M(\Delta)(z) = z^2-(n-1)z + \Delta_\partial: H_2^2(\partial \B)\to L^2(\partial \B), \qquad z\in \C .
\end{eqnarray}
Denoting by $0=\lambda_0>\lambda_1>\ldots$ the different eigenvalues of $\Delta_\partial$, 
we see that $\sigma_M(\Delta)$ is invertible for all $z\not=q_j^\pm$, where 
$$q_j^\pm = \frac{n-1}2\pm \sqrt{\left(\frac{n-1}2\right)^2-\lambda_j}.$$
The $q_j^\pm$ are simple poles of $\sigma_M(\Delta)^{-1}$ except for the case where $n=1$ and $j=0$, when $q_0^+=q_0^-=0$ is a double pole. 

In the sequel we shall assume that none of the $q_j^\pm$ equals $(n-3)/2-\gamma$. 
Then the minimal domain of $\Delta$ in 
$\cH^{s,\gamma}_p(\B)$ can be shown to be $\cH^{s+2,\gamma+2}_p(\B)$, see e.g. \cite[Proposition 2.3]{Sh}. 
The elements of $\cE$ are linear combinations of functions of the form 
$\omega(x) c_\rho(y)x^{-\rho}$,
where $\omega\in C^\infty_c([0,1))$ is a cut-off function, i.e. $\omega \ge0$ and $\go(x)=1$ for $x$ close to zero, the exponents $\rho$ run over all $q_j^\pm$ in the interval 
$$I_\gamma= \Big(\frac{n-3}{2}-\gamma, \frac{n+1}{2}-\gamma\Big),$$ 
and, for $\rho=q_j^\pm$, the $c_\rho\in C^{\infty}(\partial\mathbb{B})$ are elements of the eigenspace of $\Delta_\partial$ associated with the eigenvalue $\lambda_j$. 
If $n=1$ and $0\in I_\gamma$, then $\cE$ also contains terms of the form 
$\omega(x) c_0(y)\log(x)$, with $c_0$ in the zero eigenspace of $\Delta_\partial$. 
Therefore, every closed extension of $\Delta$ has a domain $\cH^{s+2,\gg+2}_p(\B)\oplus \underline \cE$ with a subspace $\underline \cE$ of $\cE$. 

In our specific situation, fix $\gamma$ as in \eqref{gamma}. The interval $I_\gamma$ will then include the pole $0=q_0^-$ (but none of the $q_j^-$ for $j\ge1$), 
and the space $\cE$ will contain 
the space $\mathcal E_0$ of all smooth functions which are locally constant near 
$\partial \B$. 
As in \cite{RS3} we therefore focus on the closed extension $\underline \Delta_s$ with the domain 
\begin{eqnarray}\label{dds}
\cD(\underline \Delta_s) = \cH^{s+2,\gg+2}_p(\B) \oplus \mathcal E_0.
\end{eqnarray}

For $s\in \R$, we let 
\begin{eqnarray}\label{Xs}
X^s_0 = \cH^{s,\gg}_p(\B)\quad \text{ and } \quad X^s_1 = \cD(\underline\gD_s) = \cH^{s+2,\gg+2}_p(\B)\oplus \mathcal E_0 .
\end{eqnarray}
Moreover, we introduce the real interpolation space $X^s_{\frac{1}{q},q}= (X_1^s,X_0^s)_{\frac{1}{q},q}$.
Here $1<p,q<\infty$ are taken so large that \eqref{pq1} holds 
and $s\in\mathbb{R}$ satisfies
\begin{gather}\label{sigma}
s>-1+\frac{n+1}{p}+\frac{2}{q}.
\end{gather}
This implies that, for all sufficiently small $\gve>0$, 
\begin{eqnarray}\label{int}
X^s_{\frac{1}{q},q}\hookrightarrow \mathcal{H}_{p}^{s+2-\frac{2}{q}-\varepsilon,\gamma+2-\frac{2}{q}-\varepsilon}(\mathbb{B})\oplus\mathcal{E}_0\hookrightarrow C(\B).
\end{eqnarray}
The first embedding was shown in \cite[Lemma 5.2]{RS3}; the second follows from Lemma \ref{abr}. 

We start the analysis of the PME with the following maximal $L^q$-regularity result. 

\begin{theorem}\label{elldomain2}
Let $s\in \R$, $\gt\in[0,\pi)$ and $\phi>0$. For $\gg$ satisfying \eqref{gamma}, let $\underline{\Delta}_{s}$ be the extension in $\cH^{s,\gg}_p(\B)$ with domain \eqref{dds}.
Then $c-\underline \Delta_s\in \cP(\gt)\cap\mathcal{BIP}(\phi)$ for suitably large $c>0$. In particular, $-\underline \Delta_{s} $ has maximal $L^q$-regularity.
\end{theorem}

This is \cite[Theorem 4.2]{RS3}. It follows from \cite[Theorem 3.11]{RS3} and \cite[Theorem 4.1]{RS3} and is based on \cite{CSS2}, \cite[Theorem 2.9]{RS1}, \cite[Theorem 3.3]{RS2} and \cite[Theorem 5.7]{Sh}. From Theorem \ref{elldomain2} we derived in \cite[Theorem 6.1]{RS3}:

\begin{theorem}\label{6.1}Let $s$, $\gamma$, $p$ and $q$ be chosen as in \eqref{gamma}, \eqref{pq1} and \eqref{sigma}. Let $v\in X^s_{\frac{1}{q},q}$ be strictly positive on $\B$. Then, for every $\theta\in [0,\pi)$ there exists a $c>0$ such that the operator 
$$c-v\underline\Delta_s: X_1^s\to X_0^s$$ 
is $R$-sectorial of angle $\theta$, where $\underline\Delta_s$ is the Laplacian \eqref{dds}. In particular, $-v\underline\Delta_s$ has maximal $L^q$-regularity.
\end{theorem}

By \eqref{int} the assumption that $v\in X^s_{\frac{1}{q},q}$ implies that 
\begin{eqnarray}\label{6.1.a}
v\in \cH^{s+2-\frac{2}{q}-\gve,\gg+2-\frac{2}{q}-\varepsilon}_q(\B)\oplus \mathcal E_0\quad \text{ for small }\quad \gve>0,
\end{eqnarray} 
and it is only this regularity which is needed. It implies that $v$ is continuous on $\B$, locally constant on $\partial \B$ and defines a multiplier on $X_0^s$. 

The proof of Theorem \ref{6.1} was based on the following idea: The statement is easily seen to be true when $v$ is constant. 
For $s=0$, by standard perturbation results for $R$-sectoriality it is also true if $v$ differs from a constant by a function whose supremum norm (which is the norm as a multiplier on $X_0^0$) is small. 
The proof then relied on a freezing-of-coefficients type argument, where one covers the manifold $\B$ by neighborhoods $B_j$ of given points $z_j$, $j=0, \ldots, N$, and considers the case, where $v$ is replaced by a function $v_j$ equal to $v$ near $z_j$ and equal to $v(z_j)$ outside the neighborhood. 
In this construction, each component of the boundary $\partial \B$ plays the role of a single point, and the corresponding neighborhood is a collar neighborhood of this boundary component. 
It is crucial for the above perturbation result that the neighborhoods are so small that $v_j$ differs from $v(z_j)$ by a function whose supremum norm is sufficiently small. 
Using a partition of unity, the full resolvent of $c-v\underline\Delta_0$ is finally constructed in terms of a Neumann series, and $R$-sectoriality is shown by using this representation. We proceeded then by induction and interpolation for the rest of the values of $s$. 

\begin{remark}\label{REXT}
It is easily seen that the proof of \cite[Theorem 6.1]{RS3} also shows the following. For every $\theta\in [0,\pi)$ there is a $c>0$ such that the operator $c-v\underline\Delta_s: X_1^s\to X_0^s$ is $R$-sectorial of angle $\theta$, provided that one of the following conditions holds:
\begin{itemize}
\item[(a)] $s\in\mathbb{N}$, $\gamma$ is chosen as in \eqref{gamma}, $p$, $q$ satisfy
$$
\frac{n+1}p+\frac2q<2 \quad\text{and}\quad \gamma>\frac{n-3}{2}+\frac{2}{q}
$$
and 
$$
v\in \bigcap_{\varepsilon>0}\cH^{s+2-\frac{2}{q}-\gve,\gg+2-\frac{2}{q}-\varepsilon}_p(\B)\oplus \mathcal E_0
$$ 
is strictly positive on $\B$. 
\item[(b)] $s\geq0$, $\gamma$, $p$, $q$ are chosen as \eqref{gamma} and \eqref{pq1} and 
\begin{eqnarray*}
v\in \bigcap_{\varepsilon>0}\cH^{s+1-\frac{2}{q}-\gve,\gg+2-\frac{2}{q}-\varepsilon}_p(\B)\oplus \mathcal E_0
\end{eqnarray*} 
is strictly positive on $\B$.
\end{itemize}
\end{remark}

By using Theorem \ref{6.1} and a theorem by Cl\'ement and Li, see \cite[Theorem 2.1]{CL}, in \cite[Theorem 1.1]{RS3} the following result was established:

\begin{theorem}[Short time solution]\label{sts} 
Let $s, \gamma, p$ and $q$ as in \eqref{gamma}, \eqref{pq1} and \eqref{sigma}, and let $u_0\in X^s_{\frac{1}{q},q}$ be strictly positive on $\B$. 
Moreover, let $f=f(\lambda,t)$ be a holomorphic function of $\lambda$ in a neighborhood of the range of $u_0$ with values in Lipschitz functions in $t\in [0,T_0]$, for some $T_{0}>0$. 
Then there exists some $T>0$ and a unique 
\begin{eqnarray*}\label{sts.1}
u\in W^{1,q}(0,T;X_0^s)\cap L^{q}(0,T;X_1^s)
\end{eqnarray*}
solving the Cauchy problem for the inhomogeneous porous medium equation 
$$u'(t) -\Delta (u^m(t)) = f(u(t),t), \quad u(0) = u_0.$$
\end{theorem}

\begin{remark} 
By \cite[Theorem 5.6]{RS2} $($or \cite[Theorem 4.1]{Ro2}$)$, the maximal $L^q$-regularity result of Theorem \ref{elldomain2} extends to the case of manifolds with warped conical singularities, i.e to the case when the metric $h$ in \eqref{g} depends $x$. 
The only difference is that $\lambda_{1}$ in \eqref{epsilon1} refers then to the greatest nonzero eigenvalue of the Laplacian $\Delta_{h(0)}$ on $\partial\mathbb{B}$. Moreover, the proofs of \cite[Theorem 6.1]{RS3} and \cite[Theorem 6.5]{RS3} and therefore Theorem \ref{sts} extend immediately to this case as well. 
Therefore, in the sequel we will consider always the general case of warped conical singularities. 
\end{remark}

Under the change of variables $u^{m}=w$, we obtain from \eqref{e1}-\eqref{e2} the following equivalent equations
\begin{eqnarray}\label{v1}
w'(t)-mw^{\frac{m-1}{m}}(t)\Delta w(t)&=&0, \quad t>0,\\\label{v2}
w(0)&=&w_0 = u_{0}^{m}.
\end{eqnarray}
Similarly to \cite[Theorem 6.5]{RS3}, we have the following maximal regularity result for the above problem.

\begin{theorem}\label{t22}
Let $s$, $\gamma$, $p$ and $q$ be chosen as in \eqref{gamma}, \eqref{pq1}, \eqref{sigma} and $w_{0}\in X_{\frac{1}{q},q}^{s}$ such that $w_{0}\ge c>0$ on $\mathbb{B}$. Then, there exists some $T>0$ and a unique
$$
w\in W^{1,q}(0,T;X_0^s)\cap L^{q}(0,T;X_1^s)
$$
solving \eqref{v1}-\eqref{v2}.
\end{theorem}

\begin{proof}
We apply the Theorem of Cl\'ement and Li \cite[Theorem 2.1]{CL} to the problem, which is already in the quasilinear form. We pick as usual the closed extension \eqref{dds}. 

Since $w_0\in X_{{1}/{q},q}^{s}$, it has the property \eqref{6.1.a}. By Lemma \ref{abr}, 
this is also true for $w_{0}^{(m-1)/{m}}$.
Now Remark \ref{REXT} implies that the operator $-mw_{0}^{(m-1)/{m}}\underline{\Delta}_{s}$ has maximal $L^q$-regularity. In order to apply the Clément-Li theorem, it remains to show that $w\mapsto mw^{(m-1)/{m}}\underline{\Delta}_{0}$ is a Lipschitz map from an open neighborhood of $w(0)$ in $ X_{1/{q},q}^{s}$ to $\mathcal{L}(X_{1}^{s},X_{0}^{s})$.
This follows by Lemma \ref{abr}; for details see \cite[(6.20)]{RS3}.
\end{proof}

\begin{remark}\label{equivalence}\rm 
Let $\gamma$, $p,q$ and $s$ be chosen as in \eqref{gamma}, \eqref{pq1} and \eqref{sigma}. Then the solvability of \eqref{e1}-\eqref{e2} and \eqref{v1}-\eqref{v2} in the maximal $L^q$-regularity space $W^{1,q}(0,T;X_0^s)\cap L^{q}(0,T;X_1^s)$ is equivalent for strictly positive initial values as long as the solution is strictly positive. In this case the equality $u^{m}=w$ holds on $[0,T]\times\B$:

Suppose that $u\in W^{1,q}(0,T;X_0^s)\cap L^{q}(0,T;X_1^s)$ is a strictly positive solution of \eqref{e1}-\eqref{e2}, and let $w=u^m$. We claim that 
\begin{eqnarray}\label{equivalence.2}
w\in W^{1,q}(0,T;X_0^s)\cap L^{q}(0,T;X_1^s).
\end{eqnarray}
In order to show this we first note that $u\in C([0,T]; X^s_{\frac{1}{q},q})$, and therefore the norm of $u(t)$ in the space 
$\cH^{s+2-\frac{2}{q}-\gve,\gamma+2-\frac{2}{q}-\gve}_p(\B)\oplus \mathcal E_0$ is bounded for each $\gve>0$ as $t$ varies over $[0,T]$. By Lemma \ref{abr}, this space is closed under holomorphic functional calculus and the norm of $u^{m-1}(t)$ is bounded, too. 
Moreover, since the functions in this space act as bounded multipliers on $X_0^s=\cH^{s,\gamma}_p(\B)$ according to Lemma \ref{abr}, we conclude that, for suitable constants $c_1,c_2\ge0$, 
$$
\|u^m(t)\|_{\cH^{s,\gamma}_p(\B)}
\le c_1\|u(t)\|_{\cH^{s,\gamma}_p(\B)}\|u^{m-1}(t)\|_{\cH^{s+2-\frac{2}{q}-\gve,\gamma+2-\frac{2}{q}-\gve}_p(\B)\oplus \mathcal E_0}\le c_2\|u(t)\|_{\cH^{s,\gamma}_p(\B)}.
$$
In particular, we see that 
\begin{eqnarray}\label{equivalence.1}
\int_0^T\|w(t)\|^q_{\cH^{s,\gamma}_p(\B)}\,dt
\le c_2^q \int_0^T\|u(t)\|^q_{\cH^{s,\gamma}_p(\B)}\,dt <\infty,
\end{eqnarray}
and hence $w\in L^q(0,T; \cH^{s,\gamma}_p(\B))$. A similar argument together with the identity $\partial_t w = mu^{m-1} \partial_tu$ implies that also $\partial_t w \in L^q(0,T; \cH^{s,\gamma}_p(\B))$, so that 
$w\in W^{1,q}(0,T;\cH^{s,\gamma}_p(\B))$. As $u$ solves \eqref{e1}, we see that also $\Delta w = \Delta u^m=\partial_t u \in L^q(0,T; \cH^{s,\gamma}_p(\B))$. Since $c-\underline{\Delta}_{s}: X_1^s=\cH^{s+2,\gamma+2}_p(\B)\oplus \mathcal{E}_{0} \to X_0^s=\cH^{s,\gamma}(\B)$ is invertible for some $c>0$, we have the estimate
$$
\|v\|_{X_1^s}\le c_3 \left(\|\Delta v\|_{X_0^s} + \|v\|_{X_0^s} \right)
$$
for a suitable constant $c_3$. We conclude with \eqref{equivalence.1} that for suitably large $c_4$
\begin{eqnarray*}
\left(\int_0^T \|w(t)\|^q_{X^s_1}\, dt \right)^{\frac{1}{q}}
\le c_4\left(\int_{0}^{T} \|\Delta w(t)\|^q_{X_0^s}dt\right)^{\frac{1}{q}} +c_{4}\left(\int_{0}^{T} \|w(t)\|_{X_0^s}^q\, dt\right)^{\frac{1}{q}}<\infty,
\end{eqnarray*}
hence $w\in L^q(0,T; X_1^s)$ and \eqref{equivalence.2} is established. As a consequence, $w\in C([0,T];X_{\frac{1}{q},q}^{s})$ so that $w(0) \in X_{\frac{1}{q},q}^{s}$. Finally, we note that 
$$
\partial_t w = mu^{m-1}\partial_tu = mu^{m-1}\Delta u^m = mw^{\frac{m-1}m} \Delta w,
$$
so that \eqref{v1}-\eqref{v2} holds.

Conversely suppose that $w\in W^{1,q}(0,T;X_0^s)\cap L^{q}(0,T;X_1^s)$ is a strictly positive solution of \eqref{v1}-\eqref{v2} and let $u=w^{\frac{1}{m}}$. Arguments as above show also that $u\in W^{1,q}(0,T;X_0^s)\cap L^{q}(0,T;X_1^s)$ and $u$ solves \eqref{e1}-\eqref{e2}.
\end{remark}

\section{Smoothing for the Abstract Quasilinear Parabolic Problem}

In this section we provide a smoothness result for the abstract parabolic quasilinear equation, based on maximal $L^q$-regularity theory for non-autonomous parabolic problems on Banach scales. 
This will be our main tool for showing smoothness in the space variable for the porous medium equation later on. An alternative result is stated in the Appendix. Both results are based on the ideas in \cite[Section 5]{Ro2}.

For each $i\in\mathbb{N}$, let $Y_{1}^{i}\overset{d}{\hookrightarrow}Y_{0}^{i}$ be a continuously and densely injected complex Banach couple, such that $Y_{0}^{i+1}\overset{d}{\hookrightarrow}Y_{0}^{i}$, $Y_{1}^{i+1}\overset{d}{\hookrightarrow}Y_{1}^{i}$ and $Y_{1}^{i}\overset{d}{\hookrightarrow}(Y_{1}^{i+1},Y_{0}^{i+1})_{\frac{1}{q},q}$, for some fixed $q\in(1,\infty)$. Consider the problem
\begin{eqnarray}\label{aqpp3}
u'(t)+A(u(t))u(t)&=&F(u(t),t),\quad t>0,\\\label{aqpp4}
u(0)&=&u_{0},
\end{eqnarray}
where $u_{0}\in (Y_{1}^{0},Y_{0}^{0})_{\frac{1}{q},q}$. 
\begin{theorem}[Smoothing]\label{SE}
Assume that $u_{0}\in Z$, where $Z$ is a subset of $(Y_{1}^{0},Y_{0}^{0})_{\frac{1}{q},q}$, and: 
\begin{itemize}
\item[(i)] There exists a $T>0$ and a unique 
$$
u\in W^{1,q}(0,T;Y_{0}^{0})\cap L^{q}(0,T;Y_{1}^{0})\hookrightarrow C([0,T];(Y_{1}^{0},Y_{0}^{0})_{\frac{1}{q},q})
$$ 
solving \eqref{aqpp3}-\eqref{aqpp4}, such that $u(t)\in Z$ for all $t\in[0,T]$. Moreover, $A(u(t)):Y_{1}^{0}\rightarrow Y_{0}^{0}$ has maximal $L^{q}$-regularity for each $t\in[0,T]$ and $A(u(\cdot))\in C([0,T];\mathcal{L}(Y_{1}^{0},Y_{0}^{0}))$.
\item[(ii)] For each $i\in \mathbb{N}$ and each $v$ in $Z\cap (Y_{1}^{i},Y_{0}^{i})_{\frac{1}{q},q}$, $A(v):Y_{1}^{i+1}\rightarrow Y_{0}^{i+1}$ has maximal $L^{q}$-regularity. Furthermore, $A(v(\cdot))\in C([0,T];\mathcal{L}(Y_{1}^{i+1},Y_{0}^{i+1}))$ for each $v\in C([0,T];Z\cap(Y_{1}^{i},Y_{0}^{i})_{\frac{1}{q},q})$.
\item[(iii)] For all $i\in \mathbb{N}$, $t\mapsto F(v(t),t)\in L^{q}(0,T;Y_{0}^{i+1})$ for all $v\in C([0,T];Z\cap(Y_{1}^{i},Y_{0}^{i})_{\frac{1}{q},q})$. 
\end{itemize}
Then, for each $\varepsilon\in(0,T)$ we have
$$
u\in \bigcap_{i\in\mathbb{N}} W^{1,q}(\varepsilon,T;Y_{0}^{i})\cap L^{q}(\varepsilon,T;Y_{1}^{i}).
$$ 
\end{theorem}
\begin{proof}
Take any $t_{0}\in(0,\varepsilon)$ such that $u(t_{0})\in Y_{1}^{0}$ and consider the non-autonomous linear parabolic problem
\begin{eqnarray}\label{2na1}
w'(t)+A(u(t))w(t)&=&F(u(t),t),\quad t\in(t_{0},T),\\\label{2na2}
w(t_{0})&=&u(t_{0}).
\end{eqnarray}
Clearly, the above problem has the solution
\begin{gather}\label{2wu}
u\in W^{1,q}(t_{0},T;Y_{0}^{0})\cap L^{q}(t_{0},T;Y_{1}^{0})\hookrightarrow C([t_{0},T];(Y_{1}^{0},Y_{0}^{0})_{\frac{1}{q},q}).
\end{gather}
Moreover by the assumption (i) and \cite[Theorem 2.7]{Ar} it is unique.
We proceed by describing the steps that increase the regularity of $u$.

\underline{Step 1.} By \eqref{2wu} and the assumptions (ii)-(iii), $A(u(t)):Y_{1}^{1}\rightarrow Y_{0}^{1}$ has maximal $L^{q}$-regularity for each $t\in[t_{0},T]$, $A(u(\cdot))\in C([t_{0},T];\mathcal{L}(Y_{1}^{1},Y_{0}^{1}))$ and $t\mapsto F(u(t),t)\in L^{q}(t_{0},T;Y_{0}^{1})$. 
Since $u(t_{0})\in Y_{1}^{0}\overset{d}{\hookrightarrow} (Y_{1}^{1},Y_{0}^{1})_{\frac{1}{q},q}$, by \cite[Theorem 2.7]{Ar} the problem \eqref{2na1}-\eqref{2na2} has a unique solution
\begin{gather}\label{2f}
f\in W^{1,q}(t_{0},T;Y_{0}^{1})\cap L^{q}(t_{0},T;Y_{1}^{1}).
\end{gather}
Since $Y_{j}^{1}\hookrightarrow Y_{j}^{0}$, $j\in\{0,1\}$, we infer from uniqueness that 
\begin{gather}\label{2uwf}
u=f\in W^{1,q}(t_{0},T;Y_{0}^{1})\cap L^{q}(t_{0},T;Y_{1}^{1})\hookrightarrow C([t_{0},T];(Y_{1}^{1},Y_{0}^{1})_{\frac{1}{q},q}).
\end{gather}

\underline{Step 2.} Take $t_{1}\in (t_{0},\varepsilon)$ such that $u(t_{1})\in Y_{1}^{1}$ and consider the problem 
\begin{eqnarray}\label{2na3}
\eta'(t)+A(u(t))\eta(t)&=&F(u(t),t),\quad t\in(t_{1},T),\\\label{2na4}
\eta(t_{1})&=&u(t_{1}).
\end{eqnarray}
By \eqref{2uwf} and the assumptions (ii)-(iii), $A(u(t)):Y_{1}^{2}\rightarrow Y_{0}^{2}$ has maximal $L^{q}$-regularity for each $t\in[t_{1},T]$, $A(u(\cdot))\in C([t_{1},T];\mathcal{L}(Y_{1}^{2},Y_{0}^{2}))$ and $t\mapsto F(u(t),t)\in L^{q}(t_{1},T;Y_{0}^{2})$. Thus, since $u(t_{1})\in Y_{1}^{1}\overset{d}{\hookrightarrow} (Y_{1}^{2},Y_{0}^{2})_{\frac{1}{q},q}$, by \cite[Theorem 2.7]{Ar} the problem \eqref{2na3}-\eqref{2na4} has a unique solution 
\begin{gather*}
\eta\in W^{1,q}(t_{1},T;Y_{0}^{2})\cap L^{q}(t_{1},T;Y_{1}^{2}).
\end{gather*}
Clearly,
\begin{gather*}\label{eta2}
\eta\in W^{1,q}(t_{1},T;Y_{0}^{1})\cap L^{q}(t_{1},T;Y_{1}^{1}).
\end{gather*}
Defining a function $v\in W^{1,q}(t_{0},T;Y_{0}^{1})\cap L^{q}(t_{0},T;Y_{1}^{1})$ by 
$$
v(t)=\Bigg\{\begin{array}{lll} \eta(t) &\text{when} & t\in(t_{1},T] \\ f(t) &\text{when} & t\in[t_{0},t_{1}]\end{array},
$$
by uniqueness \eqref{2f} we find that 
\begin{gather*}
u=f=\eta\in W^{1,q}(t_{1},T;Y_{0}^{2})\cap L^{q}(t_{1},T;Y_{1}^{2}).
\end{gather*}

The result now follows by iteration.
\end{proof}

\begin{remark}
A smoothness result similar to Theorem \ref{SE} holds if we replace the maximal $L^q$-regularity property (and the corresponding assumptions) with maximal continuous regularity and use \cite[Theorem 7.1]{Am4} instead of \cite[Theorem 2.7]{Ar}.
\end{remark}

\section{Properties of the Solutions of the PME}
\setcounter{equation}{0}

In order to establish the existence of a long time solution for \eqref{e1}-\eqref{e2}, we will show first certain properties of the short time solution obtained by Theorem \ref{t22}. 
We start with the following existence result related to the linear theory of the degenerate Laplacian. 

\begin{lemma}\label{phi} Let $T>0$ and $\gamma$, $p$, $q$ be chosen as \eqref{gamma}-\eqref{pq1}.
Moreover, let $\omega\in \cap_{s>0}C^{\infty}([0,T], X^s_0)$ be 
constant near the boundary $\partial\mathbb{B}$, and $\psi \in \cap_{s>0}C^{\infty}([0,T], X^s_{\frac{1}{q},q})$ with $\psi>c>0$ on $[0,T]\times\mathbb{B}$ for some $c>0$. Then the parabolic problem
\begin{eqnarray}\label{phi.1}
\partial_t\phi+\psi\Delta\phi +\omega&=&0, \quad t\in (0,T),\\
\phi(T)&=&0,\label{phi.2}
\end{eqnarray}
has a unique solution $\phi \in \cap_{\nu,s>0}W^{\nu,q}(0,T;X_1^s)$, such that $\phi\geq 0$ on $[0,T]\times\mathbb{B}$. 
\end{lemma}

\begin{proof}
Reversing time, we obtain for any $s>0$ the linear non-autonomous problem
\begin{eqnarray}\label{ANP}
\partial_tv-\zeta\underline\Delta_s v&=&\eta, \quad t\in (0,T) ,\\\label{ANP2}
v(0)&=&0,
\end{eqnarray}
where $v(t)=\phi(T-t)$, $\zeta(t) = \psi(T-t)$ and $\eta(t)=\omega(T-t)$, $t\in [0,T]$. 

According to Theorem \ref{6.1}, the operator $-\zeta(t)\underline\Delta_{s}$ has maximal regularity for each fixed $t$ in view of our assumption on $\psi$. 
Moreover, $t\mapsto \zeta(t)\underline\Delta_s$ is a smooth map from $[0,T]$ to $\cL(X_1^s,X_0^s)$ due to Lemma \ref{abr}. Hence \cite[Theorem 2.7]{Ar} implies the existence of a solution $v$ in $L^q(0,T; X_1^s)\cap W^{1,q}(0,T;X_0^s)$ for the problem \eqref{ANP}-\eqref{ANP2}. 
More precisely (see the proof of \cite[Theorem 4.2]{Ro}), there exists some $c_{0}>0$ such that the operator $A=\partial_{t}+c_{0}-\zeta\underline\Delta_{s}$ with domain $W^{1,q}(0,T;X_{0}^{s})\cap L^{q}(0,T;X_{1}^{s})$ in $L^{q}(0,T;X_{0}^{s})$ admits a bounded inverse. Here by $\partial_{t}$ we denote the operator $u\mapsto \partial_{t}u$ in $L^{q}(0,T;X_{0}^{s})$ with domain 
$$
\{u\in W^{1,q}(0,T;X_{0}^{s})\, |\, u(0)=0\}.
$$
Then we can write $\phi=e^{c_{0}t}A^{-1}(e^{-c_{0}t}\eta)$. By the assumption on $\eta$ we have that $e^{-c_{0}t}\eta\in\mathcal{D}(A^{k})$ for each $k\in\mathbb{N}$. 
Therefore, $\phi\in\mathcal{D}(A^{k})$ for each $k\in\mathbb{N}$. Especially, $\phi \in W^{\nu,q}(0,T;X_1^s)$ for all $\nu\ge0$. 
Finally, the positivity of $\phi$ follows from the maximum principle as e.g. in the proof of \cite[Lemma 3.15]{Ma} after replacing 
$\Delta$ with $\psi\Delta$.
\end{proof}

We next establish a basic property of the porous medium equation for the case of a conic manifold. 
The solutions we are going to consider are of maximal regularity either in the continuous or in the $L^q$ sense. We describe first these two cases. 
Assume that $s$, $\gamma$, $p$, $q$ are chosen as in \eqref{gamma}, \eqref{pq1} and \eqref{sigma}.
\begin{itemize}
\item[(C1)] Let $u_{0}\in X_{\frac{1}{q},q}^{s}$ and
$$
u\in W^{1,q}(0,T;X_0^s)\cap L^{q}(0,T;X_1^s)\hookrightarrow C([0,T];C(\mathbb{B}))
$$
for some $T>0$.
\item[(C2)] Let $u_{0}\in \bigcap_{\varepsilon>0}\mathcal{H}_{p}^{s+2-\frac{2}{q}-\varepsilon,\gamma+2-\frac{2}{q}-\varepsilon}(\mathbb{B})\oplus\mathcal{E}_0$ and
$$
u \in C^{1}((0,T];X_{0}^{s})\cap C((0,T];X_{1}^{s})\cap \bigcap_{\varepsilon>0} C([0,T];\mathcal{H}_{p}^{s+2-\frac{2}{q}-\varepsilon,\gamma+2-\frac{2}{q}-\varepsilon}(\mathbb{B})\oplus\mathcal{E}_0)\hookrightarrow C([0,T];C(\mathbb{B}))
$$
for some $T>0$.
\end{itemize}

\begin{theorem}[Comparison principle]\label{p1}
Let $s$, $\gamma$, $p$, $q$ be chosen as in \eqref{gamma}, \eqref{pq1} and \eqref{sigma}. 
Given initial data $u_{0,1}, u_{0,2}\in C(\B)$ satisfying $c_{0}\leq u_{0,1}\leq u_{0,2}$ on $\mathbb{B}$, for some $c_{0}>0$, let $u_{1}, u_{2}\in C([0,T];C(\mathbb{B}))$ be two solutions of \eqref{e1} for some $T>0$ with initial values $u_{0,1}$ and $u_{0,2}$, 
respectively, such that $(u_{0,1},u_{1})$ satisfies {\rm(C1)} or {\rm (C2)} and $(u_{0,2},u_{2})$ also satisfies {\rm(C1)} or {\rm(C2)}. Then $c_{0}\leq u_{1}\leq u_{2}$ on $[0,T]\times\mathbb{B}$.

In particular, if $u$ with the regularity as before is a solution of \eqref{e1} with initial value $u_{0}\in X_{\frac{1}{q},q}^{s}$ satisfying $c_{0}\leq u_{0}\leq c_{1}$ on $\mathbb{B}$ for suitable constants $c_{0},c_{1}>0$, then $c_{0}\leq u\leq c_{1}$ on $[0,T]\times\mathbb{B}$.
\end{theorem}
 
\begin{proof}
The corresponding proof in Vazquez's book, see \cite[Theorem 6.5]{Va}, adapts to our situation. We sketch the details. Note that for $\B$ with the Riemannian measure $d\mu_g$ induced by the degenerate metric $g$, no boundary terms will arise in Green's formula, i.e. we have
$$
\int_{\B} (u\Delta v+\skp{\nabla u,\nabla v}_g) \, d\mu_g = 0
$$
for $u,v$ e.g. as in (C1) or (C2). Write $\mathbb{B}_{T}=(0,T]\times \mathbb{B}$, $\partial_{t}(\cdot)=(\cdot)_{t}$ and let $\phi\in C^{1,2}(\overline{\mathbb{B}_{T}})$ be a test function, $\phi\ge0$, $\phi(T,\cdot)=0$. Subtracting we obtain from \eqref{e1} 
\begin{gather}
\int_{\mathbb{B}_{T}}\Big((u_{1}-u_{2})\phi_{t}+(u_{1}^{m}-u_{2}^{m})\Delta \phi\Big)d\mu_{g} dt\ge 0.
\end{gather}
Moreover, define the function $\alpha$ by $\alpha=0$ when $u_{1}=u_{2}$ and
$$
\alpha=\frac{u_{1}^{m}-u_{2}^{m}}{u_{1}-u_{2}} \quad \text{when} \quad u_{1}\neq u_{2}.
$$

Let $\omega$ be as in Lemma \ref{phi} and $\phi$ the solution of \eqref{phi.1}-\eqref{phi.2} with $\psi=\alpha_{\varepsilon}$, where $\alpha_{\varepsilon}$ is a smooth approximation of $\alpha$ satisfying $0\leq\varepsilon\leq\alpha_{\varepsilon}\leq K$, for some fixed $K>0$. 
By \eqref{phi.1} we estimate
\begin{gather}\label{esstt1}
\int_{\mathbb{B}_{T}}(u_{1}-u_{2})\omega d\mu_{g}dt\leq \int_{\mathbb{B}_{T}}|u_{1}-u_{2}||\alpha-\alpha_{\varepsilon}||\Delta\phi| d\mu_{g}dt=J.
\end{gather}
Moreover,
\begin{gather}\label{esdtr2}
J\leq \Big(\int_{\mathbb{B}_{T}}\alpha_{\varepsilon}(\Delta\phi)^{2} d\mu_{g}dt\Big)^{\frac{1}{2}}\Big(\int_{\mathbb{B}_{T}}\frac{|\alpha-\alpha_{\varepsilon}|^{2}}{\alpha_{\varepsilon}}|u_{1}-u_{2}|^{2} d\mu_{g}dt\Big)^{\frac{1}{2}}.
\end{gather}

Let $\zeta$ be a smooth positive function on $[0,T]$ such that $\frac{1}{2}\leq \zeta\leq 1$ and $\zeta_{t}\geq c>0$. By multiplying \eqref{phi.1} with $\zeta\Delta\phi$ and then integrating over $\mathbb{B}_{T}$ we obtain 
$$
\int_{\mathbb{B}_{T}}\zeta\phi_{t}\Delta\phi d\mu_{g}dt+\int_{\mathbb{B}_{T}}\zeta\alpha_{\varepsilon}(\Delta\phi)^{2} d\mu_{g}dt+\int_{\mathbb{B}_{T}}\zeta\omega\Delta\phi d\mu_{g}dt=0.
$$
For the first term in the above equation we have
$$
\int_{\mathbb{B}_{T}}\zeta\phi_{t}\Delta\phi d\mu_{g}dt
=-\int_{\mathbb{B}_{T}}\zeta\langle\nabla\phi_{t},\nabla\phi\rangle_{g} d\mu_{g}dt
\geq \frac{1}{2}\int_{\mathbb{B}_{T}}\zeta_{t}\langle\nabla\phi,\nabla\phi\rangle_{g} d\mu_{g}dt.
$$
Therefore
$$
 \frac{1}{2}\int_{\mathbb{B}_{T}}\zeta_{t}\langle\nabla\phi,\nabla\phi\rangle_{g} d\mu_{g}dt+\int_{\mathbb{B}_{T}}\zeta\alpha_{\varepsilon}(\Delta\phi)^{2} d\mu_{g}dt\leq \int_{\mathbb{B}_{T}}\zeta\langle\nabla\omega,\nabla\phi\rangle_{g} d\mu_{g}dt.
$$
From Cauchy's inequality we conclude that
$$
\int_{\mathbb{B}_{T}}\langle\nabla\phi,\nabla\phi\rangle_{g} d\mu_{g}dt+\int_{\mathbb{B}_{T}}\alpha_{\varepsilon}(\Delta\phi)^{2} d\mu_{g}dt\leq C_{1} \int_{\mathbb{B}_{T}}\langle\nabla\omega,\nabla\omega\rangle_{g} d\mu_{g}dt,
$$
for certain $C_{1}>0$. 
Hence we deduce from \eqref{esstt1}-\eqref{esdtr2} that
\begin{gather}\label{esstt123}
\int_{\mathbb{B}_{T}}(u_{1}-u_{2})\omega d\mu_{g}dt
\leq \sqrt{C_1}\Big(\int_{\mathbb{B}_{T}}\langle\nabla\omega,\nabla\omega\rangle_{g} d\mu_{g}dt,\Big)^{\frac12}\Big(\int_{\mathbb{B}_{T}}\frac{|\alpha-\alpha_{\varepsilon}|^{2}}{\alpha_{\varepsilon}}|u_{1}-u_{2}|^{2} d\mu_{g}dt\Big)^{\frac{1}{2}}.
\end{gather}

By \eqref{contint} and Lemma \ref{abr} each $u_{i}$, $u_{i}^{m}$, $i=1,2$, belongs to 
$$
\bigcap_{\varepsilon>0}C([0,T];\mathcal{H}_{p}^{s+2-\frac{2}{q}-\varepsilon,\gamma+2-\frac{2}{q}-\varepsilon}(\mathbb{B})\oplus\mathcal{E}_0)\hookrightarrow C([0,T];C(\B))
$$
and therefore is square integrable on $\mathbb{B}_{T}$ with respect to $d\mu_{g}dt$. 
If we knew that $0<\varepsilon\le\alpha\le K$, we could deduce immediately from 
\eqref{esstt123} that $u_1\le u_2$ by letting $\alpha_\varepsilon \to \alpha$.
In the general case, one uses an approximation argument, which is the same as in the classical case. 
\end{proof}

From the above comparison principle we obtain the following analog for the problem \eqref{v1}-\eqref{v2}.

\begin{theorem}\label{comp2}
Let $s$, $\gamma$, $p$, $q$ be chosen as in \eqref{gamma}, \eqref{pq1} and \eqref{sigma}. 
Given initial data $w_{0,1}, w_{0,2}\in C(\B)$ satisfying $c_{0}\leq w_{0,1}\leq w_{0,2}$ on $\mathbb{B}$, for some $c_{0}>0$, let $w_{1},w_{2}$ be two solutions of \eqref{v1} 
with initial values $w_{0,1}$ and $w_{0,2}$, respectively, such that $(w_{0,1},w_{1})$ satisfies {\rm (C1)} or {\rm (C2)} and $(w_{0,2},w_{2})$ also satisfies {\rm (C1)} or {\rm (C2)}. 
Then $c_{0}\leq w_{1}\leq w_{2}$ on $[0,T]\times\mathbb{B}$.

In particular, if $w$ with the regularity as before is a solution of \eqref{v1} with initial value $w_{0}\in X_{\frac{1}{q},q}^{s}$ satisfying $c_{0}\leq w_{0}\leq c_{1}$ on $\mathbb{B}$ for suitable constants $c_{0}, c_{1}>0$, then $c_{0}\leq w\leq c_{1}$ on $[0,T]\times\mathbb{B}$.
\end{theorem}
\begin{proof}
We examine the three cases separately. Assume first that both $(w_{0,1},w_1)$ and $(w_{0,2},w_2)$ satisfy (C1). By continuity, $w_1(t)$ and $w_2(t)$ will be strictly positive for some time. We claim that in fact both will be $\ge c_0$ for all $t\in [0, T]$. 
Indeed, let $\epsilon>0$ and suppose that there is a first time $\tau_0\in(0,T]$ such that, say, $w_1(\tau_0)$ attains the value $c_0-\epsilon>0$. By Remark \ref{equivalence}, $u_1=w_1^{\frac{1}{m}}$ and $u_2=w_2^{\frac{1}{m}}$ furnish solutions of \eqref{e1} on $[0,\tau_0]\times \mathbb B$ with initial data $w_{0,1}^{\frac{1}{m}},w_{0,2}^{\frac{1}{m}}\ge c_0^{\frac{1}{m}}$. 
The comparison principle then implies that $u_1(\tau_0)\ge c_0^{\frac{1}{m}}$ and hence $w_1(\tau_0)\ge c_0$, which is a contradiction. We can now apply Remark \ref{equivalence} once more to conclude that $c_0\le w_1\le w_2$ on $[0,T]\times \mathbb B$.

Next, assume that both $(w_{0,1},w_1)$ and $(w_{0,2},w_2)$ satisfy (C2). We claim again that both $w_1$ and $w_2$ will be $\ge c_0$ for all $t$. Indeed suppose that for some $\epsilon>0$ one of them attains the value $c_0-2\epsilon>0$. Choose the largest $\tau_0\in (0,T)$ such that $w_1(t), w_2(t)\ge c_0-\epsilon$ for $0\le t\le \tau_0$. As $w_1(\tau_0)$ and $w_2(\tau_0)$ both belong to $X_1^s$, we find short time solutions
$$
v_1,v_2\in W^{1,q}(\tau_0,\tau_0+T_0; X^s_0) \cap L^q(\tau_0,\tau_0+T_0; X_1^s)
$$
for some $T_0>0$, solving \eqref{e1} with initial values at $\tau_0$ given by $w_{1}^{\frac{1}{m}}(\tau_0)$ and $w_{2}^{\frac{1}{m}}(\tau_0)$, respectively. Note that both belong to $X_1^s$, since this space is closed under holomorphic functional calculus. The comparison principle, Theorem \ref{p1}, implies that $v_1(t),v_2(t) \ge (c_0-\epsilon)^{\frac{1}{m}}$ for $\tau_0\le t\le \tau_0+T_0$. Then $v_1^m$ and $v_2^m$ belong to 
$$
W^{1,q}(\tau_0,\tau_0+T_0; X^s_0) \cap L^q(\tau_0,\tau_0+T_0; X_1^s)
$$
by Remark \ref{equivalence} and solve \eqref{v2} with initial values $w_1(\tau_0)$ and $w_2(\tau_0)$, respectively. By uniqueness of Theorem \ref{t22}, $v_1^m=w_1$ and $v_2^m=w_2$. Consequently, $w_1(t), w_2(t)\ge c_0-\epsilon$ for $\tau_0\le t\le \tau_0+T_0$, contradicting the maximality of $\tau_0$. 

Lemma \ref{abr} implies that 
$$
w_1^{\frac1m}, w_2^{\frac1m}, w_1^{\frac{1-m}m},w_2^{\frac{1-m}m}\in C((0,T];X_{1}^{s})\cap \bigcap_{\varepsilon>0}
C([0,T]; \mathcal{H}^{s+2-\frac{2}{q}-\varepsilon, \gamma+2-\frac{2}{q}-\varepsilon}_p(\mathbb{B})).
$$
From the identity $\partial_{t}(w_{1}^{\frac{1}{m}}) = \frac{1}{m}w_{1}^{\frac{1-m}{m}}\partial_{t}w_{1}$ and Lemma \ref{abr} we conclude that $w_1^{\frac{1}{m}}\in C^1((0,T]; X_0^s)$. 
Hence $w_1^{\frac{1}{m}}$ and, by the same argument, $w_2^{\frac{1}{m}}$ belong to the class (C2). 
Moreover, they solve \eqref{e1} with initial values $w_{0,1}^{\frac1m}$ and $w_{0,2}^{\frac1m}$, respectively. Hence the comparison principle Theorem \ref{p1} implies the assertion. 

Finally, assume that $(w_{0,1},w_{1})$ satisfy (C1) and $(w_{0,2},w_{2})$ satisfy (C2). By the first step, $c_{0}\leq w_{1}$ on $[0,T]\times\B$ and hence, by Remark \ref{equivalence} we have that $(w_{0,1}^{\frac{1}{m}},w_{1}^{\frac{1}{m}})$ is a (C1) solution to \eqref{e1}-\eqref{e2} on $[0,T]\times\B$. Similarly, by the second step $c_{0}\leq w_{2}$ on $[0,T]\times\B$ and $(w_{0,2}^{\frac{1}{m}},w_{2}^{\frac{1}{m}})$ is a (C2) solution of \eqref{e1}-\eqref{e2} on $[0,T]\times\B$. Then, the result follows by Theorem \ref{p1} applied to $(w_{0,1}^{\frac{1}{m}},w_{1}^{\frac{1}{m}})$ and $(w_{0,2}^{\frac{1}{m}},w_{2}^{\frac{1}{m}})$.
\end{proof}

We denote by $(g^{ij})=(g_{ij})^{-1}$ and $(h^{ij})=(h_{ij})^{-1}$ the inverses of the metric tensors of $g$ and $h$ in local coordinates. In local coordinates $(z^1,\ldots, z^{n+1})$, the gradient $\nabla u$ of a scalar function $u$ is defined as
$$
\nabla u = \sum_{i,j} g^{ij}\frac{\partial u}{\partial z^i} \frac{\partial}{\partial z^j}
$$ 
and the divergence of a vector field $F$ is given by 
$$
\Div F = \sqrt{\det [g]}^{-1} \sum_{i} \frac{\partial(F^i\sqrt{\det [g]}) }{\partial z^i}.
$$
Note that
\begin{eqnarray}\label{divergence}
\Div(uF) = u\Div F + \skp{\nabla u, F}_g \quad \text{where}\quad \skp{\nabla u, F}_g= \sum_{i} \frac{\partial u}{\partial z^i} F^i.
\end{eqnarray} 
In local coordinates $(x,y)$ close to the boundary 
\begin{gather*}
\langle \nabla u,\nabla v\rangle_{g}=\frac{1}{x^{2}}\Big((x\partial_{x}u)(x\partial_{x}v)+\sum_{i,j}h^{ij}\frac{\partial u}{\partial y^{i}} \frac{\partial v}{\partial y^{j}} \Big).
\end{gather*} 
Since $\Delta=\Div\nabla$, we also have 
\begin{gather}\label{Dm}
\Delta(u^{m})=\Div\nabla u^m = m\Div (u^{m-1}\nabla u).
\end{gather}
This allows us to rewrite the porous medium equation \eqref{e1}-\eqref{e2} in the form 
\begin{eqnarray}\label{e3}
u'(t)-m \Div(u^{m-1}\nabla u)&=&0, \quad t>0,\\
\quad u(0)&=&u_{0}.
\end{eqnarray} 

\begin{theorem}[Hölder continuity] \label{Holder}
Let $s$, $\gamma$, $p$, $q$ be chosen as in \eqref{gamma}, \eqref{pq1} and \eqref{sigma}, $u_0\in X^s_{\frac{1}{q},q}$ be strictly positive and let $u\in W^{1,q}(0,T;X_0^s)\cap L^{q}(0,T;X_1^s)$ be a solution of \eqref{e1} with initial value $u_0$. 
Then $u$ is $\alpha$-Hölder continuous on $\B$, uniformly in $[0,T]$ 
and $\alpha/2$-Hölder continuous on $[0,T]$, uniformly in $\B$, 
for some $\alpha>0$. 
Both $\alpha$ and the Hölder norms of $u$ depend only on the dimension $n$, the initial value $u_0$ and the metric $g$. 
\end{theorem}

This is a consequence of a result by Lady\v zenskaya, Solonnikov and Ural'ceva. 
In Section II.8 of \cite{LSU} they introduce a class 
$\widehat{\mathfrak B}(U\times [0,T), M,\tau, r,\delta, b)$ of functions satisfying 
certain integral inequalities. 
Here $U$ is a subset of $\R^{n+1}$, $M=\sup_{U\times [0,T)} |u|$, and $\tau, r,\delta$ and $b$ are real parameters. 
These integral inequalities, stated in II.(7.1) and II.(7.2) and, equivalently, in II.(7.5)
of \cite{LSU} for a function $u$ are formulated 
in terms of the functions $(\pm u-\kappa)^+ = \max\{(\pm u-\kappa), 0\} $ and a cut-off function $\varphi$. Here $\kappa\in \R$ is arbitrary with $\sup u-\kappa\le \delta$ (correspondingly $\sup (-u)-\kappa\le \delta$); the supremum is taken over the integration domain.

It then turns out that the elements of 
$\widehat{\mathfrak B}(U\times [0,T), M,\tau, r,\delta, b)$ are Hölder continuous 
(for the precise definition of the Hölder spaces and their norms see 
\cite[(1.10), (1.11), (1.12)]{LSU}). Moreover, Theorem II.8.2 states that both the 
Hölder exponent and the Hölder norm of the functions 
on a smaller space time cylinder are determined by the parameters 
$M,\tau, r,\delta$ and $b$. 

The proof, below, is inspired by Amann's proof of \cite[Lemma 5.1]{Am3}. 
We shall show that every solution to the porous medium equation 
with initial value $u_0$ satisfying $0<M_{\min}\le u_0\le M_{\max}$ belongs to the class $\widehat{\mathfrak B}(U\times [0,T), M,\tau, r,\delta, b)$ 
with $M=M_{\max}$, $r=2(n+3)/(n+1)$, $b = 2/(n+1)$; $\delta$ and $\tau$
will be chosen suitably, depending only on $M_{\min}, M_{\max}$ and $\B$. 

\begin{proof}
{\def\uph{u^+_\kappa}
We choose a function $\gvp\in C^\infty_c( U\times [0,T))$ with values in $[0,1]$ 
and a constant $\kappa>0$. Then we write $\uph = (u-\kappa)^+$ (where $\xi^{+}=\max\{\xi,0\}$, $\xi\in\R$), multiply the solution $u$ of the porous medium equation with $\gvp^2\uph$ and integrate over $\B\times [t_0,t_1]$, $0\le t_0\le t_1<T$, with the volume measure $d\mu_g$ with respect to $g$ on $\B$. Using that $\partial \uph$ equals $\partial u$ on $\{u>\kappa\}$ and $0$ else, for arbitrary derivatives $\partial$ with respect to space or time, we find 
\begin{eqnarray}
\lefteqn{\int_{t_0}^{t_1} \int_\B \gvp^2 \uph\, \partial_t u \, d\mu_{g} dt 
=
\int_{t_0}^{t_1} \int_\B \gvp^2\frac12 \partial_t(\uph)^2\, d\mu_{g} dt \nonumber}\\
&=&
\frac12 \int_{t_0}^{t_1} \int_\B \partial_t(\gvp^2(\uph)^2)\, d\mu_{g} dt
- \int_{t_0}^{t_1} \int_\B\gvp \partial_t\gvp (\uph)^2\, d\mu_{g} dt\nonumber\\
&=&\frac12 \|\gvp\uph\|^2_{L^2(\B)}|_{t_0}^{t_1} 
- \int_{t_0}^{t_1} \int_\B\gvp \partial_t\gvp (\uph)^2\, d\mu_{g} dt\label{h.0}.
\end{eqnarray}
Moreover, by \eqref{divergence} and the divergence theorem (denoting the exterior normal in $\partial\B$ by $\nu$) 
\begin{eqnarray}
\lefteqn{\int_{t_0}^{t_1} \int_\B \gvp^2\, \uph m\Div ( u^{m-1} \nabla u) \, d\mu_{g} dt }
\nonumber\\
&=& \int_{t_0}^{t_1} \int_\B \Div(\gvp^2\, \uph m u^{m-1} \nabla u) \, d\mu_{g} dt 
-\int_{t_0}^{t_1} \int_\B \skp{\nabla(\gvp^2\, \uph), m u^{m-1} \nabla u }_g \, d\mu_{g} dt \nonumber\\
&=& \int_{t_0}^{t_1} \int_{\partial\B} \skp{\gvp^2\, \uph m u^{m-1} \nabla u,\nu}_g \, d\mu_{h} dt \label{h.1}\\
&&-2\int_{t_0}^{t_1} \int_\B \gvp \, \uph m u^{m-1} \skp{ \nabla \gvp,\nabla u}_g \, d\mu_{g} dt 
\label{h.2}\\
&&- \int_{t_0}^{t_1} \int_\B \gvp^2 m u^{m-1}\skp{\nabla \uph, \nabla u}_g \, d\mu_{g} dt. 
\label{h.3}
\end{eqnarray}

In the integral \eqref{h.1}, $d\mu_h$ denotes the induced volume measure on $\partial \B$. The function $\gvp^2\, \uph u^{m-1}$ is continuous up to $\partial \B$ and thus bounded.
The normal vector $\nu$ is also bounded; it coincides with $\partial_x$. The gradient $\nabla u$ belongs to $\cH^{s+1-\frac{2}{q}-\varepsilon, \gamma +1-\frac{2}{q}-\varepsilon}_p(\B)$, for any $\varepsilon>0$. By \cite[Corollary 2.5]{RS1} it is therefore $O(x^{-1+\eta})$ for $\eta=\gamma- \frac{n-3}{2}-\frac{2}{q}>0$. The volume measure, on the other hand, is $\sqrt{\det[h]}x^n\, dy$, $n\ge1$. Hence this integral vanishes. 

Concerning the integral \eqref{h.3} we notice that 
$$
\skp{\nabla \uph,\nabla u}_g = \skp{\nabla \uph,\nabla \uph}_g= |\nabla \uph|_g^2\ge c |\nabla \uph|_{\text{eucl}}^2
$$
for a suitable constant $c$ depending only on $\mathbb B$; here $|v|_{\text{eucl}}$ is the Euclidean norm of a vector, which appears in \cite[II.(7.5)]{LSU}. We next recall that the comparison principle Theorem \ref{p1} implies that 
$M_{\min}\le u(z,t)\le M_{\max}$ for all $z\in \B$ and $t_0\le t\le t_1$. 

Hence the integral \eqref{h.3} can be estimated from below by 
$c_{0}\|\gvp\, \nabla \uph\|^2_{L^2(\B\times[t_0,t_1])}$ for a suitable constant $c_{0}= c_{0}(u_0,c)$. Bringing this term to the left hand side, we obtain
with \eqref{h.0}
\begin{eqnarray}
\lefteqn{\frac12 \|\gvp\uph\|^2_{L^2(\B)}|_{t_0}^{t_1} 
+ c_{0} \|\gvp\, \nabla \uph\|^2_{L^2(\B\times[t_0,t_1])}\nonumber}\\
&\le& \int_{t_0}^{t_1} \int_\B\gvp |\partial_t\gvp| (\uph)^2\, d\mu_{g} dt\label{h.4}\nonumber\\
&&+2\int_{t_0}^{t_1} \int_\B \gvp \, \uph m u^{m-1} \left|\skp{ \nabla \gvp,\nabla u}_g \right|\, d\mu_{g} dt .\label{h.5}
\end{eqnarray} 
In order to estimate the integral \eqref{h.5} we first notice that $\uph\nabla u = \uph\nabla \uph$ and that, by Cauchy's inequality, 
$$ 
\gvp\, \uph\left|\skp{\nabla \gvp,\nabla \uph}_g\right| \le \sigma \gvp^2|\nabla\uph|^2 + \frac1{4\sigma} |\nabla\gvp|^2(\uph)^2
$$
for arbitrary $\sigma>0$. Choose $\sigma$ so small that $\sigma mu^{m-1} < \frac{c_{0}}4.$ Then 
\begin{eqnarray*}
\lefteqn{2\int_{t_0}^{t_1} \int_\B \gvp \, \uph m u^{m-1}| \skp{ \nabla \gvp,\nabla u}_g| \, d\mu_{g} dt }\\
&\le& \frac{c_{0}}2 \int_{t_0}^{t_1} \int_\B \gvp^2|\nabla \uph|^2 \,d\mu_{g} dt
+ \frac{c_{0}}{8\sigma^{2}}\int_{t_0}^{t_1} \int_\B |\nabla \gvp|^2(\uph)^2\, d\mu_{g} dt.
\end{eqnarray*}
Combining this with \eqref{h.5}, we obtain
\begin{eqnarray*}
\lefteqn{\frac12 \|\gvp\uph\|^2_{L^2(\B)}|_{t_0}^{t_1} 
+ \frac{c_{0}}{2} \|\gvp\, \nabla \uph\|^2_{L^2(\B \times[t_0,t_1])}\nonumber}\\
&\le&(\frac{c_{0}}{8\sigma^{2}}+1)\int_{t_0}^{t_1} \int_\B(\gvp|\partial_t\gvp| + |\nabla \gvp|^2)(\uph)^2\, d\mu_{g} dt.
\end{eqnarray*}
Hence the function $u$ satisfies the estimates \cite[II.(7.5)]{LSU}. In the same way, the estimates are satisfied with $u$ replaced by $-u$. Therefore $u$
belongs to the class $\widehat{\mathfrak B}(U \times[0,T), M,\tau, r,\delta, b)$
with the above choices of the parameters.
} 
\end{proof}

An immediate consequence of the above H\"older estimate is the uniform boundedness in $T$ of the $R$-sectorial bound of the linearized term in \eqref{v1}, as we can deduce from the following result. 

\begin{proposition}\label{Rbound}
Let $s=0$ and $\gamma$, $p$, $q$, be chosen as in \eqref{gamma} and \eqref{pq1}. Let $w\in W^{1,q}(0,T;X_0^0)\cap L^q(0,T; X_1^0)$ be the short time solution to the problem \eqref{v1}-\eqref{v2} with initial value $w_{0}\in X_{\frac{1}{q},q}^{0}$, $c^{-1}\leq w_{0}\leq c$, $c>0$, which exists by Theorem $\ref{t22}$ for some $T>0$. For fixed $t\in [0,T)$ consider the operator $\alpha_{t}\underline{\Delta}_{0}: X_{1}^{0}\rightarrow X_{0}^{0}$, where $\alpha_{t}=mw^{\frac{m-1}{m}}(t)$. Then, for every $\theta\in[0,\pi)$ there exists some $c_{0}>0$, independent of $t\in[0,T)$ and $T$ such that $c_{0}-\alpha_{t}\underline{\Delta}_{0}$ is $R$-sectorial of angle $\theta$ with $R$-bound uniformly bounded in $t$ and $T$.
\end{proposition}

\begin{proof}
By Theorem \ref{p1}, we have that $c^{-1}\leq w(t)\leq c$ for all $t\in[0,T]$. By Remark \ref{equivalence}, \eqref{e1}-\eqref{e2} also has a unique solution $u\in W^{1,q}(0,T;X_0^0)\cap L^q(0,T; X_1^0)$ which satisfies $c^{-\frac{1}{m}}\leq u(t)\leq c^{\frac{1}{m}}$ for all $t\in[0,T]$ and is given by $u=w^{\frac{1}{m}}$.

We follow the idea in the sketch of the proof of Theorem \ref{6.1}, see the proof of \cite[Theorem 6.1]{RS3} for details. For $s=0$ the base space $X_0^0$ is the weighted $L^p$-space $\cH^{0,\gamma}_p(\B)$. 
Hence the norm of a function as a multiplier on $X_0^0$ is just the sup-norm. We know from Theorem \ref{Holder} that $u$ is Hölder continuous and that both the Hölder exponent and the Hölder norm depend only on $u_0$. 
As a consequence, $\alpha_t$ is also Hölder continuous, and its Hölder norm is independent of $t$ and $T$. 
It follows that the oscillation of $\alpha_t$ on one of the neighborhoods $B_j$ mentioned in the sketch will be small, provided the diameter is small (for $B_0$: provided the width of the collar is small). 
Hence the whole construction for the proof can be carried out uniformly in $t$ and $T$ and therefore the shift $c_{0}$ and the $R$-bound can be taken uniform in $t$ and $T$. 
\end{proof}

We proceed by establishing interpolation space estimates for the short time solution $w$. Our proof is based on a maximal $L^q$-regularity inequality for linear non-autonomous parabolic problems, which is obtained by operator valued functional calculus estimates. It is based on the proof of \cite[Theorem 4.2]{Ro} and the proof of \cite[Theorem 5.5]{Ro2}, where an application of this approach to a semilinear problem is demonstrated. 

\begin{theorem}[Maximal $L^q$-regularity space estimates]\label{xq}
Take $s=0$, $p$, $q$ as in \eqref{gamma}, \eqref{pq1} and $w_{0}\in X_{\frac{1}{q},q}^{0}$ strictly positive. There exists a positive function $C(\cdot)\in C(\mathbb{R})$ with the following property{\em:} for any $T>0$ such that there exists a unique solution $w\in W^{1,q}(0,T;X_0^0)\cap L^{q}(0,T;X_1^0)$ of \eqref{v1}-\eqref{v2} with initial data $w_{0}$ we have that
$$
\|w\|_{W^{1,q}(0,T;X_0^0)\cap L^{q}(0,T;X_1^0)}\leq C(T).
$$
\end{theorem}

\begin{proof}
Write $A(t)=c_{0}-mw^{\frac{m-1}{m}}(t)\underline{\Delta}$ with $c_{0}>0$ as in Proposition \ref{Rbound} and let $\{e^{-tA(0)}\}_{t\geq0}$, be the bounded holomorphic semigroup generated by $A(0)$. According to \cite[Proposition III.4.10.3]{Am} there exist $K,\eta>0$, independent of $T$ such that 
\begin{gather}\label{ppl}
\|A(0)e^{-tA(0)}w_0\|_{L^{q}(0,T;X_{0}^{0})}\leq K\|w_0\|_{X_{\frac{1}{q},q}^{0}}.
\end{gather}

For $c\geq0$ consider the linear non-autonomous parabolic problem
\begin{eqnarray}\label{u12}
v'(t)+(c+A(t)) v(t)&=&e^{-ct}(A(0)-A(t))A(0)^{-1}A(0)e^{-tA(0)}w_{0},\quad t\in(0,T),\\\label{u22}
v(0)&=&0,
\end{eqnarray} 
Since $w$ solves \eqref{v1}, the above problem has the solution
\begin{gather}
e^{-ct}(e^{-c_{0}t}w-e^{-tA(0)}w_{0})\in W^{1,q}(0,T;X_{0}^{0})\cap L^{q}(0,T;X_{0}^{1})\hookrightarrow C([0,T]; \mathcal{H}_{p}^{2-\frac{2}{q}-\varepsilon,\gamma+2-\frac{2}{q}-\varepsilon}(\mathbb{B})\oplus\mathcal{E}_0),
\end{gather}
for each $\gve>0$.

Moreover, when $\gve>0$ is sufficiently small, by Lemma \ref{abr} we have that
\begin{gather}\label{w5}
w,w^{\frac{m-1}{m}}\in C([0,T];\mathcal{H}_{p}^{2-\frac{2}{q}-\varepsilon,\gamma+2-\frac{2}{q}-\varepsilon}(\mathbb{B})\oplus\mathcal{E}_0)\hookrightarrow C([0,T];C(\mathbb{B}))
\end{gather}
and multiplication with $mw^{\frac{m-1}{m}}$ induces a bounded map in $X_{0}^{0}$, whose norm is bounded by 
$$
\|mw^{\frac{m-1}{m}}\|_{C([0,T];C(\B))}.
$$
We conclude from \eqref{w5} that 
$$
c_{0}-mw^{\frac{m-1}{m}}\underline{\Delta}_{0}\in C([0,T];\mathcal{L}(X_{1}^{0},X_{0}^{0})). 
$$
Thus, by \cite[Theorem 2.7]{Ar} and Proposition \ref{Rbound}, there exists a unique solution $v\in W^{1,q}(0,T;X_{0}^{0})\cap L^{q}(0,T;X_{0}^{1})$ and therefore
\begin{gather}\label{v}
v=e^{-ct}(e^{-c_{0}t}w-e^{-tA(0)}w_{0}).
\end{gather} 

In order to obtain maximal $L^q$-regularity space estimates, we regard $v$ as the action of the inverse to the operator $\partial_{t}+A(\cdot)+c$ on the right hand side of \eqref{u12}. This inverse has been constructed in the proof of \cite[Theorem 4.2]{Ro} for $c>0$ sufficiently large. For convenience we describe the construction, which relies on a freezing-of-coefficients argument; details can be found in \cite{Ro}: Denote by $B$ the operator $u\mapsto \partial_{t}u$ in $Y_{0}=L^{q}(0,T;X_{0}^{0})$ with domain
$$
\mathcal{D}(B)=Y_{2}=\{u\in W^{1,q}(0,T;X_{0}^{0})\, |\, u(0)=0\}.
$$
Since the Banach space $X_{0}^{0}$ satisfies the UMD property, by \cite[Theorem 8.5.8]{Ha}, for any $\phi\in(0,\frac{\pi}{2})$ the operator $B$ admits a bounded $H^\infty$-calculus of angle $\phi$ (concerning this property, see e.g. \cite[Definition 2.9]{DHP}). Furthermore, by \cite[Corollary 8.5.3 (a)]{Ha} the $H^\infty$-bound of $B$ is uniformly bounded in $T$. We regard $A(\cdot)$ as an operator in $L^{q}(0,T;X_{0}^{0})$ with domain $Y_{1}=L^{q}(0,T;X_{1}^{0})$. 

For small $r\in (0,1)$ choose $0<t_1<\ldots <t_N$ such that the intervals $(t_j-r,t_j+r)$, $j=1,\ldots ,N$, cover $[0,T]$. 
Choose a subordinate partition of unity $\phi_1,\ldots,\phi_N$ and smooth functions 
$\chi_j$, $j=1\ldots,N$, supported in $(t_j-r,t_j+r)$ 
such that $\chi_j\phi_j=\phi_j$. 

By the Kalton-Weis theorem \cite[Theorem 4.5]{KW1}, each $A(t_{j})+B$ belongs to $\mathcal{P}(0)$. Define 
$$
A_{j}=\chi_{j}(A(\cdot)+B)+(1-\chi_{j})(A(t_{j})+B)=A(t_{j})+B+\chi_{j}(A(\cdot)-A(t_{j})),
$$
with $\mathcal{D}(A_{j})=Y_{1}\cap Y_{2}$ in $Y_{0}$ for each $j$. 
Taking $r$ sufficiently small (and possibly $N$ large), we can achieve that 
$\|\chi_{j}(A(\cdot)-A(t_{j}))(A(t_{j})+B+c)^{-1}\|_{\mathcal{L}(Y_{0})}\leq 1/2$;
this is a consequence of the continuity of $A(\cdot)$ and the uniform boundedness of the norm of $A(\xi)(A(\xi)+B+c)^{-1}$ in $\mathcal{L}(Y_{0})$ for $(\xi,c)\in[0,T]\times \R_{+}$. By standard sectoriality perturbation theory $A_{j}$ belongs again to $\mathcal{P}(0)$; in fact
\begin{gather*}
(A_{j}+c)^{-1}=(A(t_{j})+B+c)^{-1}\sum_{k=0}^{\infty}(-1)^{k}\big(\chi_{j}(A(\cdot)-A(t_{j}))(A(t_{j})+B+c)^{-1}\big)^{k}.
\end{gather*}
We can then construct a left and right inverse to $A(\cdot)+B+c$ for $c\geq0$ sufficiently large as in \cite[(4.5)]{Ro} and \cite[(4.6)]{Ro}, respectively. Due to the H\"older estimate in Theorem \ref{Holder}, we can choose $c$ uniformly bounded, whenever $T$ is bounded. Furthermore, the $\mathcal{L}(L^{q}(0,T;X_{0}^{0}))$-norm of $A(\cdot)(B+A(\cdot)+c)^{-1}$ can be estimated by $C_{0}(T)$, for some positive function $C_{0}(\cdot)\in C(\mathbb{R})$ (see \cite[(4.5)]{Ro}). Therefore, by \eqref{ppl} and \eqref{u12}-\eqref{u22} we estimate 
\begin{eqnarray}\nonumber
\lefteqn{\|v\|_{W^{1,q}(0,T;X_{0}^{0})\cap L^{q}(0,T;X_{1}^{0})}}\\\label{maxreg3}
&\leq& C_{1}(T)\|e^{-ct}(A(0)-A(\cdot))A(0)^{-1}A(0)e^{-tA(0)}w_{0}\|_{L^{q}(0,T;X_{0}^{0})}\leq C_{2}(T)\|w_{0}\|_{X_{\frac{1}{q},q}^{0}},
\end{eqnarray}
for some positive functions $C_{1}(\cdot),C_{2}(\cdot)\in C(\mathbb{R})$. 

Standard maximal $L^q$-regularity for the problem
\begin{eqnarray*}\label{eta12}
f'(t)+A(0) f(t)&=&0,\quad t\in(0,T),\\\label{eta22}
f(0)&=&w_{0},
\end{eqnarray*} 
implies that 
\begin{gather}\label{yy}
\|e^{-tA(0)}w_{0}\|_{W^{1,q}(0,T;X_{0}^{0})\cap L^{q}(0,T;X_{1}^{0})}\leq C_{3}\|w_{0}\|_{X_{\frac{1}{q},q}^{0}},
\end{gather}
for some $C_3>0$ (see e.g. \cite[(2.9)]{Ro2}). From \eqref{v}, \eqref{maxreg3} and \eqref{yy} we obtain 
\begin{eqnarray*}\
\lefteqn{\|w\|_{W^{1,q}(0,T;X_{0}^{0})\cap L^{q}(0,T;X_{1}^{0})}}\\
&\leq& (c+c_{0}+1)e^{(c+c_{0})T}\Big(\|v\|_{W^{1,q}(0,T;X_{0}^{0})\cap L^{q}(0,T;X_{1}^{0})}+\|e^{-tA(0)}w_{0}\|_{W^{1,q}(0,T;X_{0}^{0})\cap L^{q}(0,T;X_{1}^{0})}\Big)\\
&\leq& (c+c_{0}+1)e^{(c+c_{0})T}(C_{2}(T)+C_{3})\|w_{0}\|_{X_{\frac{1}{q},q}^{0}}.
\end{eqnarray*}
\end{proof}

\begin{corollary}\label{rmrk}
Estimate \eqref{contint} implies that for suitable $0<\delta_{0}<\delta_{1}<\infty$ in Theorem \ref{xq}, we also have $\|w(t)\|_{X_{\frac{1}{q},q}^{0}}\leq c\, {C}(T)$ with some constant $c$ independent of $T\in[\delta_{0},\delta_{1}]$.
\end{corollary}

\section{Proof of Theorem 1.1}
We are now in the position to show existence and uniqueness of the long time solution for \eqref{e1}-\eqref{e2} as well as smoothness in space and time. The proof of long time existence is based on a successive application of a short time Banach fixed point argument by H. Amann. Due to the necessity to control the interpolation space norm (Corollary \ref{rmrk} above), we do this step for the case of $s=0$, i.e. on the weighted $L^p$-space. Then, smoothness will follow by an application of Theorem \ref{SE}.

Let us first note that also for $s_0<0$ we may reduce to the case of $s=0$:
By Theorem \ref{t22} there exists some $T_{0}>0$ and a unique 
$$
\widetilde{w}\in W^{1,q}(0,T_{0};X_{0}^{s_{0}})\cap L^{q}(0,T_{0};X_{1}^{s_{0}})$$
solving the problem \eqref{v1}-\eqref{v2} with initial data $\widetilde{w}(0)=u_{0}^{m}$. By Theorem \ref{comp2} we have $c_{0}^{m}\leq \widetilde{w}\leq c_{1}^{m}$ on $[0,T_{0}]\times\B$. Take $\xi\in(0,T_{0})$ arbitrarily close to zero such that $
\widetilde{w}(\xi)\in X_{1}^{s_{0}}$; in particular, $\widetilde w\in X_{\frac{1}{q},q}^{s_{0}+\frac{2}{q}-\varepsilon}$
for all $\varepsilon>0$ by Remark \ref{rem1.2}(b). By Theorem \ref{t22} we can start the evolution \eqref{v1}-\eqref{v2} again with initial value $\widetilde{w}(\xi)$ and $s_{0}$ replaced by $s_{0}+\frac{2}{q}-\varepsilon$. After finitely many steps, by uniqueness, see Theorem \ref{t22}, we can achieve $\widetilde{w}(t_{0})\in X_{1}^{0}$ for some $t_{0}\in(0,T_{0})$ arbitrarily close to zero. By Theorem \ref{t22}, the problem \eqref{v1} with initial data $w(t_{0})=\widetilde{w}(t_{0})\in X_{1}^{0}$ has a unique solution 
\begin{gather}\label{firstw}
w\in W^{1,q}(t_{0},t_{0}+T;X_{0}^{0})\cap L^{q}(t_{0},t_{0}+T;X_{1}^{0})\hookrightarrow C([t_{0},t_{0}+T];X_{\frac{1}{q},q}^{0})
\end{gather}
for some $T>0$. By uniqueness of Theorem \ref{t22}, we have that $\widetilde{w}=w$ on $[t_{0},\min\{T_{0},t_{0}+T\}]$. 

\subsubsection*{Long time existence.} \rm
Let $T_{\max}$ be the supremum of all such $T$ and assume that $T_{\max}<\infty$. By Theorem \ref{comp2}
\begin{gather}\label{compw}
c_{0}^{m}\leq w\leq c_{1}^{m}\quad \text{on} \quad [t_{0},t_{0}+T_{\max})\times\B.
\end{gather}
We apply \cite[Theorem 7.1]{Am4} on maximal continuous regularity solutions to abstract quasilinear parabolic equations to \eqref{v1}-\eqref{v2}.
Concerning the data \cite[(Q1)-(Q2)]{Am4} we choose $E_{0}=X_{0}^{0}$, $E_{1}=X_{1}^{0}$, $E_{\alpha}=X_{\frac{1}{q},q}^{0}$ and $E_{\beta}=X_{\frac{1}{q}+\delta,q}^{0}$ with $\delta$ defined in \eqref{delta1}, such that the following embedding holds
\begin{gather}\label{w35}
 X_{\frac{1}{q},q}^{0}\hookrightarrow X_{\frac{1}{q}+\delta,q}^{0}\hookrightarrow \mathcal{H}_{p}^{2-\frac{2}{q}-2\delta-\varepsilon,\gamma+2-\frac{2}{q}-2\delta-\varepsilon}(\mathbb{B})\oplus\mathcal{E}_0\hookrightarrow C(\mathbb{B}),
\end{gather}
for any $\varepsilon>0$ small enough. Concerning the open set in Amann's condition (Q1), we choose 
$$
V=\bigcup_{t\in[t_{0},t_{0}+T_{\max})}\Big\{v\in X_{\frac{1}{q},q}^{0}\, |\, \|v-w(t)\|_{ X_{\frac{1}{q},q}^{0}}<\frac{c_{0}^{m}\widetilde{\delta}}{C_{p,q,\gamma}}\Big\}.
$$
Here $\widetilde{\delta}\in(0,\frac{1}{2})$ and by $C_{p,q,\gamma}$ we denote the norm of the embedding $X_{\frac{1}{q},q}^{0}\hookrightarrow C(\mathbb{B})$. Then we take $V_{\alpha}=V\cap E_{\alpha}$ and $V_{\beta}=V\cap E_{\beta}$, equipped with the topologies of $E_{\alpha}$ and $E_{\beta}$, respectively. Finally, we let $A(w)=-mw^{\frac{m-1}{m}}\underline{\Delta}_{0}$.

Due to Corollary \ref{rmrk}, $V$ is bounded in $X_{\frac{1}{q},q}^{0}$. Furthermore, by \eqref{compw} and the estimate
$$
\|v-w(t)\|_{C(\mathbb{B})}\leq C_{p,q,\gamma} \|v-w(t)\|_{ X_{\frac{1}{q},q}^{0}}\leq \frac{c_{0}^{m}}{2},
$$ 
valid for any $v\in V$ and certain $t$ depending on $v$, we see that all elements in $V$ are uniformly bounded away from zero. 

{\em Step 1:} We recall from Lemma \ref{abr} that $\mathcal{H}_{p}^{2-\frac{2}{q}-2\delta-\varepsilon,\gamma+2-\frac{2}{q}-2\delta-\varepsilon}(\mathbb{B})\oplus\mathcal{E}_0$ is a Banach algebra for $\varepsilon>0$ sufficiently small and closed under holomorphic functional calculus. Moreover, the elements of this space act as multipliers on $X^0_0$.
With \eqref{w35} and \eqref{poly} it is then easy to see that 
$A(\cdot): V_{\beta}\rightarrow \mathcal{L}(X_{1}^{0},X_{0}^{0})$ is a Lipschitz map whose Lipschitz bound is estimated by the total $X_{\frac{1}{q}+\delta,q}^{0}$-bound of $V$; see \cite[(6.20)]{RS3} for details. 

{\em Step 2 :} For $\sigma>(n+1)/p$ and $\gamma>(n+1)/2$, the space $\cH^{\sigma, \gamma}_p(\B)$ 
embeds into the Zygmund space $C^{\sigma-(n+1)/p}_*(\B)$; this is shown e.g. in
the proof of \cite[Corollary 3.3]{RS3}.
In view of Corollary \ref{rmrk} and \eqref{w35} we see that the elements of $V$ are uniformly Hölder continuous for some small Hölder exponent. Just as in the proof of Proposition \ref{Rbound} we conclude that there exists a constant $c>0$ such that $c-A(v)$ is $R$-sectorial of angle $\theta$, for some fixed $\theta\in(\frac{\pi}{2},\pi)$, uniformly in $v\in\{w(t)\, |\, t\in[t_{0},t_{0}+T_{\max})\}$. Then, by taking $\widetilde{\delta}$ sufficiently small and using the perturbation result \cite[Theorem 1]{KuWe}, the same holds true uniformly in $v\in V$. As a consequence, $c-A(v)\in \mathcal{P}(\theta)$, uniformly in $v\in V$, and hence $A(v)\in \mathcal{H}(E_{1},E_{0})$ uniformly in $v\in V$, where $\mathcal{H}$ denotes generation of an analytic semigroup (see \cite[Section 4]{Am4}). 
Therefore, the above uniformity ensures that the map $v\mapsto A(v)$ from $V_{\alpha}$ to $\mathcal{H}(E_{1},E_{0})$ is locally regularly bounded in the sense of \cite[(Q2)]{Am4}. 

We conclude that both \cite[(Q1)-(Q2)]{Am4} are fulfilled and hence by \cite[Theorem 7.1]{Am4} with initial data $w(\tau_{0})\in X_{\frac{1}{q},q}^{0}$, for any 
$\tau_{0}\in[t_{0},t_{0}+T_{\max})$, there exists a unique 
\begin{eqnarray*}\label{wt}
\lefteqn{w_{\tau_{0}}\in C^{1}((\tau_{0},\tau_{0}+T_{V}];X_{0}^{0})}\\
&&\cap C((\tau_{0},\tau_{0}+T_{V}];X_{1}^{0})\cap C([\tau_{0},\tau_{0}+T_{V}];X_{\frac{1}{q},q}^{0})\cap C^{\delta}([\tau_{0},\tau_{0}+T_{V}];X_{\frac{1}{q}+\delta,q}^{0})
\end{eqnarray*}
solving \eqref{v1}-\eqref{v2} with initial data $w(\tau_{0})$, where $T_{V}$ is independent of $\tau_{0}$. Theorem \ref{comp2} implies that each $w_{\tau_{0}}$ coincides with $w$ on $[\tau_{0},\min\{\tau_{0}+T_{V},t_{0}+T_{\max}\})$. 
Taking $\tau_{0}$ close to $T_{\max}$ we obtain a contradiction. Regarding $\widetilde{w}$ as an extension of $w$ to $[0,t_{0})$, we obtain for the unique solution $w$ of \eqref{v1}-\eqref{v2} the regularity 
\begin{eqnarray}\label{s0w}
w\in C^{1}((0,T];X_{0}^{0})\cap C((0,T];X_{1}^{0})\cap W^{1,q}(0,T;X_{0}^{s_{0}}) \cap L^{q}(0,T;X_{1}^{s_{0}}),
\end{eqnarray}
valid for arbitrary $T>0$.

\subsubsection*{Smoothness in space.} We show that the solution $w$ given by \eqref{s0w} becomes instantaneously smooth in the Mellin-Sobolev spaces. 
We apply Theorem \ref{SE} to the problem \eqref{v1}-\eqref{v2} on $[\tau,T]$, $\tau\in(0,T)$, 
for some fixed $T$, with initial data $w(\tau)$ and the following setting: 
$Y_{0}^{i}=\mathcal{H}_{p}^{\frac{i}{q},\gamma}(\mathbb{B})$, $Y_{1}^{i}=\mathcal{H}_{p}^{\frac{i}{q}+2,\gamma+2}(\mathbb{B})\oplus\mathcal{E}_0$ and $A(v)u=-mv^{\frac{m-1}{m}}\underline{\Delta}_{\frac{i}{q}}u$. Furthermore, we choose the set
$$
Z=\left\{u\in (Y_{1}^{0},Y_{0}^{0})_{\frac{1}{q},q}\,\, | \,\, c_{0}^{m}\leq u(z)\leq c_{1}^{m} \quad \text{for all} \quad z\in\mathbb{B}\right\}.
$$
By Remark \ref{rem1.2}(b), 
$Y_{1}^{i}\hookrightarrow(Y_{1}^{i+1},Y_{0}^{i+1})_{\frac{1}{q},q}$. 
By \eqref{s0w} there exists a unique 
\begin{gather}\label{uniquew}
w\in W^{1,q}(\tau,T;Y_{0}^{0})\cap L^{q}(\tau,T;Y_{1}^{0}) \hookrightarrow C([\tau,T];(Y_{1}^{0},Y_{0}^{0})_{\frac{1}{q},q})
\end{gather}
solving \eqref{v1}-\eqref{v2}. 
By \eqref{compw} we have that $w(t)\in Z$ for each $t\in[\tau,T]$. 
According to Proposition \ref{Rbound}
$-mw^{\frac{m-1}{m}}(t)\underline{\Delta}_{0}$ has maximal $L^q$-regularity for each $t\in[0,T]$. 
Finally, from Lemma \ref{abr} 
we deduce that $-mw^{\frac{m-1}{m}}(t)\underline{\Delta}_{0}\in C([\tau,T];\mathcal{L}(Y_{1}^{0},Y_{0}^{0}))$, so that the condition (i) of Theorem \ref{SE} is fulfilled.

If $i\geq1$ and $v\in (Y_{1}^{i},Y_{0}^{i})_{\frac{1}{q},q}$, then
$$
v\in(Y_{1}^{i},Y_{0}^{i})_{\frac{1}{q},q}\hookrightarrow \mathcal{H}_{p}^{\frac{i-2}{q}+2-\varepsilon,\gamma+2-\frac{2}{q}-\varepsilon}(\mathbb{B})\oplus\mathcal{E}_0
$$
for all $\varepsilon>0$ according to Remark \ref{rem1.2}(b). By Lemma \ref{abr}, we also have
$$
v^{\frac{m-1}{m}}\in \mathcal{H}_{p}^{\frac{i-2}{q}+2-\varepsilon,\gamma+2-\frac{2}{q}-\varepsilon}(\mathbb{B})\oplus\mathcal{E}_0,
$$
and multiplication by $v^{\frac{m-1}{m}}$ defines a bounded map in $Y_{0}^{i+1}$. Therefore, $A(v)$ maps $Y_{1}^{i+1}$ to $Y_{0}^{i+1}$. Moreover, part (b) of Remark \ref{REXT} with $s=\frac{i+1}{q}$ implies maximal $L^q$-regularity for each $A(v)$.

Next, take
\begin{eqnarray*}
v\in\left\{u\in C([\tau,T];(Y_{1}^{i},Y_{0}^{i})_{\frac{1}{q},q}) \,\, |\,\, u(t)\in Z \quad \text{for all}\quad t\in[\tau,T]\right\}.
\end{eqnarray*}
In view of \eqref{int} and Lemma \ref{abr}, it holds that
\begin{eqnarray*}
v^{\frac{m-1}{m}}\in
\Big\{u\in \bigcap_{\varepsilon>0}C([\tau,T];\mathcal{H}_{p}^{\frac{i-2}{q}+2-\varepsilon,\gamma+2-\frac{2}{q}-\varepsilon}(\mathbb{B})\oplus\mathcal{E}_0) \,\, |\,\, 
u(t)\in [c_{0}^{m-1},c_{1}^{m-1}] \quad \text{for all}\quad t\in[\tau,T]\Big\}.
\end{eqnarray*}
Hence
$$
\|v^{\frac{m-1}{m}}(t)\underline{\Delta}_{\frac{i+1}{q}}-v^{\frac{m-1}{m}}(t')\underline{\Delta}_{\frac{i+1}{q}}\|_{\mathcal{L}(Y_{1}^{i+1},Y_{0}^{i+1})}\leq C\|v^{\frac{m-1}{m}}(t)-v^{\frac{m-1}{m}}(t')\|_{\mathcal{L}(Y_{0}^{i+1})} 
$$
for certain $C>0$, which by Lemma \ref{abr} implies that $A(v(\cdot))\in C([\tau,T];\mathcal{L}(Y_{1}^{i+1},Y_{0}^{i+1}))$, and condition (ii) of the theorem is also satisfied. 

Therefore, Theorem \ref{SE} implies that 
\begin{gather}\label{regw}
w\in \bigcap_{s>0} W^{1,q}(\tau,T;X_{0}^{s}) \cap L^{q}(\tau,T;X_{1}^{s})\hookrightarrow \bigcap_{s>0} C([\tau,T];X_{\frac{1}{q},q}^{s}).
\end{gather}

Next, we apply again \cite[Theorem 7.1]{Am4} with 
$E_{0}=X_{0}^{s_{\nu}}$, $E_{1}=X_{1}^{s_{\nu}}$, 
$E_{\alpha}=X_{\frac{1}{q},q}^{s_{\nu}}$, 
$E_{\beta}=X_{\frac{1}{q}+\delta,q}^{s_{\nu}}$, 
$V_{\alpha}=V\cap E_{\alpha}$ and 
$V_{\beta}=V\cap E_{\beta}$, where now $V$ is an open ball in 
$X_{\frac{1}{q},q}^{s_{\nu}}$ with center $w(\nu)\in X_{\frac{1}{q},q}^{s_{\nu}}$ 
for some $\nu\in[0,T)$. Here we take, with $s_0$ defined in \eqref{pq1},
$$
s_{\nu}=\Big\{\begin{array}{lll} s_{0} & \text{if} & \nu=0\\
\text{arbitrary in $(0,\infty)$} & \text{if} & \nu>0.
\end{array}
$$

By Step 1 and Step 2 above both assumptions \cite[(Q1)-(Q2)]{Am4} are fulfilled, and hence \cite[Theorem 7.1]{Am4} with initial data $w(\nu)$ implies the existence of some $\widetilde{T}_{\nu}>0$ and a unique 
$$
\eta_{\nu}\in C^{1}((\nu,\nu+\widetilde{T}_{\nu}];X_{0}^{s_\nu})
\cap C((\nu,\nu+\widetilde{T}_{\nu}];X_{1}^{s_\nu})
\cap C([\nu,\nu+\widetilde{T}_{\nu}];X_{\frac{1}{q},q}^{s_\nu})
\cap C^{\delta}([\nu,\nu+\widetilde{T}_{\nu}];X_{\frac{1}{q}+\delta,q}^{s_\nu})
$$ 
solving \eqref{v1}-\eqref{v2} on $[\nu,\nu+\widetilde{T}_{\nu}]\times\B$. By Theorem \ref{comp2} and \eqref{s0w} $\eta_{0}$ coincides with $w$ on $[0,\min\{\widetilde{T}_{0},T\}]\times\B$. Similarly, by Theorem \ref{comp2} together with \eqref{regw} we deduce that $\eta_{\nu}$ coincides with $w$ on $[\nu,\min\{\nu+\widetilde{T}_{\nu},T\}]\times\B$ for all $\nu\in(0,T)$. Therefore, we obtain the following regularity

\begin{eqnarray}\nonumber
\lefteqn{w\in C^{1}((0,T];\mathcal{H}_{p}^{s,\gamma}(\mathbb{B}))\cap 
C((0,T];\mathcal{H}_{p}^{s+2,\gamma+2}(\mathbb{B})\oplus\mathcal{E}_0)}\\\label{furtreg}
&&
\cap\, C^{\delta}((0,T];\mathcal{H}_{p}^{s,\gamma+2-\frac{2}{q}-2\delta-\varepsilon}(\mathbb{B})\oplus\mathcal{E}_0)\cap C^{\delta}([0,T];\mathcal{H}_{p}^{s_{0}+2-\frac{2}{q}-2\delta-\varepsilon,\gamma+2-\frac{2}{q}-2\delta-\varepsilon}(\mathbb{B})\oplus\mathcal{E}_0).
\end{eqnarray}
valid for all $s>0$ and all $\varepsilon>0$.

By \eqref{compw}, Remark \ref{equivalence}, the Banach algebra property of $\mathcal{H}_{p}^{s+2-\frac{2}{q}-2\delta-\varepsilon,\gamma+2-\frac{2}{q}-2\delta-\varepsilon}(\mathbb{B})\oplus\mathcal{E}_0$, for $s\geq s_{0}$ and $\varepsilon>0$ small enough and Lemma \ref{abr} the same regularity as in \eqref{s0w} and \eqref{furtreg} holds for $u=w^{\frac{1}{m}}$, which solves the original problem \eqref{e1}-\eqref{e2}; here concerning the $C^{1}((0,T];\mathcal{H}_{p}^{s,\gamma}(\mathbb{B}))$ regularity, we note that $\partial_{t}w^{\frac{1}{m}}=\frac{1}{m}w^{\frac{1-m}{m}}w'$ and use the fact that by Lemma \ref{abr} we have 
$$
w^{\frac{1-m}{m}}\in C((0,T];\mathcal{H}_{p}^{s+2,\gamma+2}(\mathbb{B})\oplus\mathcal{E}_0),
$$
and hence it acts by multiplication as a bounded map on $C((0,T];\mathcal{H}_{p}^{s,\gamma}(\mathbb{B}))$.

Moreover, 
\begin{gather}
u^{m}=w\in C^{\delta}([\nu,T];\mathcal{H}_{p}^{s_{\nu}+2-\frac{2}{q}-2\delta-\varepsilon,\gamma+2-\frac{2}{q}-2\delta-\varepsilon}(\mathbb{B})\oplus\mathcal{E}_0)
\end{gather}
for all $\varepsilon>0$, and hence 
$$
\partial_{t}u=\underline{\Delta}_{s_{\nu}}u^{m}\in C^{\delta}([\nu,T];\mathcal{H}_{p}^{s_{\nu}-\frac{2}{q}-2\delta-\varepsilon,\gamma-\frac{2}{q}-2\delta-\varepsilon}(\mathbb{B})),
$$
so that 
$$
u\in C^{1+\delta}([\nu,T];\mathcal{H}_{p}^{s_{\nu}-\frac{2}{q}-2\delta-\varepsilon,\gamma-\frac{2}{q}-2\delta-\varepsilon}(\mathbb{B})).
$$
We conclude that in addition to the regularity of \eqref{s0w} and \eqref{furtreg} for the solution $u$ we also have that
\begin{eqnarray}\nonumber
u\in C^{1+\delta}((0,T];\mathcal{H}_{p}^{s,\gamma-\frac{2}{q}-2\delta-\varepsilon}(\mathbb{B}))\cap C^{1+\delta}([0,T];\mathcal{H}_{p}^{s_{0}-\frac{2}{q}-2\delta-\varepsilon,\gamma-\frac{2}{q}-2\delta-\varepsilon}(\mathbb{B})),
\end{eqnarray}
for any $s>0$ and any $\varepsilon>0$.

\subsubsection{Smoothness in time.} By differentiating $\Delta u^{m}$ with respect to time we find that
\begin{gather}\label{poo}
\partial_{t}(\Delta u^{m})=m\Delta(u^{m-1}u').
\end{gather}
We know from Lemma \ref{abr}(c) and the space regularity obtained in the previous step that 
both $u^{m-1}$ and $u^{m-2}$ belong to $\bigcap_{s,\varepsilon>0}C((0,T];\mathcal{H}_{p}^{s,\gamma+2-\frac{2}{q}-\varepsilon}(\mathbb{B})\oplus\mathcal{E}_0).
$
From Lemma \ref{abr}(b) we obtain
that both $u^{m-1}u'$ and $u^{m-2}u'$ are elements of $\bigcap_{s>0}C((0,T];\mathcal{H}_{p}^{s,\gamma}(\mathbb{B}))$. 
Hence by \eqref{e1} and \eqref{poo} 
$$
\partial_{t}^{2}u=m\Delta (u^{m-1}u')\in \bigcap_{s>0}C((0,T];\mathcal{H}_{p}^{s,\gamma-2}(\mathbb{B})).
$$
We conclude that
$$
u\in \bigcap_{s>0}C^{2}((0,T];\mathcal{H}_{p}^{s,\gamma-2}(\mathbb{B})).
$$

By differentiating \eqref{poo} we further obtain
\begin{gather}\label{poo2}
\partial_{t}^{2}(\Delta u^{m})=m\Delta((m-1)u^{m-2}(u')^{2}+u^{m-1}u'').
\end{gather}
We write
$$
u^{m-2}(u')^{2}=u^{m-2}(\chi u')(\chi^{-1} u'),
$$
where $\chi$ is a smooth function on $\mathbb{B}$ such that $\chi>0$ on $\mathbb{B}\backslash([0,\frac{1}{2})\times\partial\mathbb{B})$ and $\chi=x^{2}$ on $[0,\frac{1}{2})\times\partial\mathbb{B}$, and observe that 
$$
u^{m-2}\chi u'\in \bigcap_{s>0}C((0,T];\mathcal{H}_{p}^{s,\gamma+2}(\mathbb{B}))
$$
and
$$
\chi^{-1} u'\in \bigcap_{s>0}C((0,T];\mathcal{H}_{p}^{s,\gamma-2}(\mathbb{B})).
$$
Therefore, both $u^{m-2}(u')^{2}$ and $u^{m-1}u''$ belong to $\bigcap_{s>0}C((0,T];\mathcal{H}_{p}^{s,\gamma-2}(\mathbb{B})),
$
so that \eqref{poo2} implies 
$$
u\in \bigcap_{s>0}C^{3}((0,T];\mathcal{H}_{p}^{s,\gamma-4}(\mathbb{B})).
$$
By applying successively the above argument we obtain that
$$
 u\in C^{k}((0,T];\mathcal{H}_{p}^{s,\gamma-2(k-1)}(\mathbb{B})) \quad \text{for all}\quad k\in\mathbb{N}\backslash\{0\}.
$$
\mbox{\ } \hfill $\square$

\section{Weak Solutions}
Our long time result provides a solution for strictly positive initial data. However, initial values that are possibly vanishing on some part of $\mathbb{B}$ also lead to solutions but in a weaker sense. In this concept, similarly to \cite[Section 3]{AP}, we recall the notion of a weak solution to \eqref{e1}-\eqref{e2}: 

\begin{definition}\label{defweak}
For $T>0$, let $\mathbb{B}_{T}=\mathbb{B}\times (0,T]$. A function $v:\overline{\mathbb{B}}_{T}\rightarrow \mathbb{R}_{+}$ is called a {\em weak solution} of \eqref{e1}-\eqref{e2} if the following are satisfied.
\begin{itemize}
\item[(i)] $\nabla v^{m}$ exists in the sense of distributions in $\mathbb{B}_{T}$ and 
$$
\int_{\mathbb{B}_{T}}(|v|^{2}+\langle\nabla v^m,\nabla v^m\rangle_g)dtd\mu_{g}<\infty,
$$
where $\nabla$, $\langle\cdot,\cdot\rangle_g$ and $d\mu_{g}$, respectively, denote the gradient, the scalar product and the Riemannian measure induced by the metric $g$.
\item[(ii)] For any $\phi\in C^{1}(\mathbb{B}_{T})$ such that $\phi=0$ on 
$\mathbb{B}\times \{T\}$, 
$$
\int_{\mathbb{B}_{T}}(\langle\nabla\phi,\nabla v^{m}\rangle_{g}-(\partial_{t}\phi)v)dtd\mu_{g}=\int_{\mathbb{B}}\phi(z,0)v(z,0)d\mu_{g}.
$$
\end{itemize}
\end{definition}

According to the above definition, we end up with the following result for non-negative valued initial data. 

\begin{theorem}[Existence of a weak solution]
Let $m\ge 1$, $s=0$ and $\gamma$, $p$, $q$ as in \eqref{gamma}, \eqref{pq1}. Furthermore, let
\[
u_{0}\in(\mathcal{H}_{p}^{2,\gamma+2}(\mathbb{B})\oplus\mathcal{E}_0,\mathcal{H}_{p}^{0,\gamma}(\mathbb{B}))_{\frac{1}{q},q}
\]
be non-negative on $\mathbb{B}$. Then, for any $T>0$, the porous medium equation \eqref{e1}-\eqref{e2} possesses a unique weak solution $u$ on $\mathbb{B}_T$.
\end{theorem}
\begin{proof}
We adapt the argument of Aronson and Peletier \cite[Appendix]{AP}. We choose a sequence $\delta_k\rightarrow 0$, $\delta_k>0$, and define $w_{0,k}=(u_0+\delta_k)^m$. For each $T>0$, Theorem \ref{Th1} and Remark \ref{equivalence} guarantee the existence of a solution $w_k$ satisfying 
\begin{eqnarray}\label{v13}
w_{k}'(t)-mw_{k}^{\frac{m-1}{m}}(t)\Delta w_{k}(t)&=&0, \quad t\in(0,T),\\\label{v23}
w_{k}(0)&=&w_{0,k}.
\end{eqnarray}
Moreover, there exists some $M>0$ such that $\delta^m_{k}\leq w_{k}\leq M$ and $w_{k+1}\leq w_{k}$ on $\mathbb{B}_T$ for all $k\in\mathbb{N}$. We can therefore let 
$w(x,t) = \lim w_k(x,t)$ and $u(x,t) = w^{\frac{1}{m}}(x,t)$. By multiplying \eqref{v13} by $m^{-1}w_{k}^{\frac{1-m}{m}}w'_{k}$ and integrating over $\mathbb{B}_T$ we obtain
\begin{eqnarray*}
\lefteqn{0\leq\frac{1}{m}\int_{\mathbb{B}_T}(w'_{k})^{2}w_{k}^{\frac{1-m}{m}}d\mu_{g}dt=\int_{\mathbb{B}_T}w'_{k}\Delta w_{k}d\mu_{g}dt}\\
&=&-\int_{\mathbb{B}_T}\skp{\nabla w'_{k}, \nabla w_{k}}_g d\mu_{g}dt=
\frac12\int_{\mathbb{B}}\skp{\nabla w_{0,k}, \nabla w_{0,k}}_g d\mu_{g}-
\frac12\int_{\mathbb{B}}\skp{\nabla w_{k}(T), \nabla w_{k}(T)}_g d\mu_{g}.
\end{eqnarray*}
Therefore
\begin{eqnarray*}
\int_{\mathbb{B}}\skp{\nabla w_{k}(t), \nabla w_{k}(t)}_g d\mu_{g}\leq\int_{\mathbb{B}}\skp{\nabla w_{0,k}, \nabla w_{0,k}}_g d\mu_{g} \quad \text{for all} \quad t\in[0,T].
\end{eqnarray*}
Let us check that the right hand side is bounded uniformly in $k$: 
In local coordinates close to the boundary
\begin{eqnarray*}
\partial_x w_{0,k} &=& m (u_0+\delta_k)^{m-1}\partial_x u_0\\
\partial_{y_j} w_{0,k} &=& m (u_0+\delta_k)^{m-1}\partial_{y_j} u_0.
\end{eqnarray*}
Hence 
\begin{eqnarray*}
\skp{\nabla w_{0,k},\nabla w_{0,k}}_g = m^2(u_0+\delta_k)^{2m-2}\skp{\nabla u_0,\nabla u_0}_g.
\end{eqnarray*}
Moreover
\begin{eqnarray*}\lefteqn{\int_{\B} \langle \nabla u_0,\nabla u_0\rangle_g \, d\mu_g}\\
&=&\int_{\B}( (\partial_x u_0)^2 + x^{-2} \sum_{ij} h^{ij} (\partial_{y_i}u_0)(\partial_{y_j}u_0)\, )d\mu_g\\
&\le& C\|u_0\|^2_{\cH^{1,1}_2(\B)}\le C' \|u_0\|^2_{\cH_p^{2-2/q-\varepsilon,\gamma+2-2/q-\varepsilon}(\B)\oplus \mathcal E_0}
\end{eqnarray*}
for suitable constants $C, C'>0$ and $\varepsilon>0$ small enough.
Thus $\|\nabla w_k(t)\|_{L^2(\mathbb B)}$ is bounded, uniformly in $k$ and $t\in [0,T]$.
The boundedness of $\{\|\nabla w_k(t)\|_{L^2(\mathbb B)}\}_k$ in $L^2(\B_T)$ implies that there exists a subsequence of $\{\nabla w_k\}$, w.l.o.g. $\{\nabla w_k\}$, converging weakly in $L^2(\B_T)$ to a limit $\psi$. The identity 
$$
\int_{\B_T}\nabla w_k \phi\, d\mu_gdt = -\int_{\B_T} w_k\nabla \phi\, d\mu_gdt
$$
valid for $\phi\in C_c^\infty(\B\times (0,T))$, 
together with dominated convergence implies that 
$$
\int_{\B_T}\psi\phi \, d\mu_gdt = -\int_{\B_T} w\nabla \phi\, d\mu_gdt,
$$
so that $\psi=\nabla w$ in the sense of distributions. 
Thus condition (i) in Definition \ref{defweak} is satisfied. Similarly, we obtain condition (ii) of Definition \ref{defweak}.
\end{proof}

\section{Appendix}

We provide here an alternative smoothness result. Let $\{Y_{i}\}_{i\in\mathbb{N}}$ be a sequence of complex Banach spaces such that 
$$
Y_{0}\overset{d}{\hookleftarrow} Y_{1} \overset{d}{\hookleftarrow} Y_{2}\overset{d}{\hookleftarrow} \ldots ,
$$
and $(Y_{i+2},Y_{i})_{\frac{1}{q},q}\overset{d}{\hookleftarrow} Y_{i+1}$, for all $i\in\mathbb{N}$ and some $q\in(1,\infty)$. Consider the abstract quasilinear problem
\eqref{aqpp3}-\eqref{aqpp4} with $u_{0}\in (Y_{2},Y_{0})_{\frac{1}{q},q}$.

\begin{theorem}[Smoothing]
Assume that $u_{0}\in Z$, where $Z$ is a subset of $(Y_{2},Y_{0})_{\frac{1}{q},q}$. Furthermore, assume:
\begin{itemize}
\item[(i)] There exists a $T>0$ and a unique 
$$
u\in W^{1,q}(0,T;Y_{0})\cap L^{q}(0,T;Y_{2})\hookrightarrow C([0,T];(Y_{2},Y_{0})_{\frac{1}{q},q})
$$ 
solving \eqref{aqpp3}-\eqref{aqpp4}, such that $u(t)\in Z$ for all $t\in[0,T]$. Moreover, $A(u(t)):Y_{2}\rightarrow Y_{0}$ has maximal $L^{q}$-regularity for each $t\in[0,T]$ and $A(u(\cdot))\in C([0,T];\mathcal{L}(Y_{2},Y_{0}))$.
\item[(ii)] For each $i\in \mathbb{N}$ and each $v\in Z\cap (Y_{i+2},Y_{i})_{\frac{1}{q},q}$, $A(v):Y_{i+3}\rightarrow Y_{i+1}$ has maximal $L^{q}$-regularity. Furthermore, $A(v(\cdot))\in C([0,T];\mathcal{L}(Y_{i+3},Y_{i+1}))$ for each $v\in C([0,T];Z\cap(Y_{i+2},Y_{i})_{\frac{1}{q},q})$. 
\item[(iii)] For each $i\in \mathbb{N}$, $t\mapsto F(v(t),t)\in L^{q}(0,T;Y_{i+1})$ for all $v\in C([0,T];Z\cap(Y_{i+2},Y_{i})_{\frac{1}{q},q})$. 
\end{itemize}
Then, for every $\varepsilon\in(0,T)$ we have
$$
u\in \bigcap_{i\in\mathbb{N}} W^{1,q}(\varepsilon,T;Y_{i}).
$$ 
\end{theorem}
\begin{proof}
The proof is similar to that of Theorem \ref{SE}. 
\end{proof}

\subsection*{Acknowledgment} We thank Roland Schnaubelt and Christoph Walker for their help and Deutsche Forschungsgemeinschaft for support through grant SCHR 319/9-1 within the program `Geometry at Infinity'.


\begin{thebibliography}{99}

\bibitem{Am4} H. Amann. {\em Dynamic theory of quasilinear parabolic equations. I. Abstract evolution equations}. Nonlinear Anal. {\bf 12}, no. 9, 895--919 (1988).

\bibitem{Am3} H. Amann. {\em Dynamic theory of quasilinear parabolic systems. III. Global existence}. Math. Z. {\bf 202}, no. 2, 219--250 (1989).

\bibitem{Am} H. Amann. {\em Linear and quasilinear parabolic problems}. Monographs in Mathematics {\bf 89}, Birkh\"auser Verlag (1995).

\bibitem{Ar} W. Arendt, R. Chill, S. Fornaro, C. Poupaud. {\em $L^{p}$-maximal regularity for non-autonomous evolution equations}. J. Differential Equations {\bf 237}, no. 1, 1--26 (2007). 

\bibitem{AP} D. Aronson, L. Peletier. {\em Large time behaviour of solutions of the porous medium equation in bounded domains}. J. Differential Equations {\bf 39}, no. 3, 378--412 (1981).

\bibitem{Ver} E. Bahuaud, B. Vertman. {\em Long-time existence of the edge Yamabe flow}. arXiv:1605.03935.

\bibitem{CL} P. Clément, S. Li. {\em Abstract parabolic quasilinear equations and application to a groundwater flow problem}. Adv. Math. Sci. Appl. {\bf 3}, Special Issue, 17--32 (1993/94).

\bibitem{CSS2} S. Coriasco, E. Schrohe, J. Seiler. {\em Bounded imaginary powers for elliptic differential operators on manifolds with conical singularities}. Math. Z. {\bf 244}, 235--269 (2003).

\bibitem{DHP} R. Denk, M. Hieber, J. Pr\"uss. {\em $R$-boundedness, Fourier multipliers, and problems of elliptic and parabolic type}. Memoirs of the American Mathematical Society {\bf 166}, no. 788, Oxford University Press (2003).

\bibitem{DV} G. Dore, A. Venni. {\em On the closedness of the sum of two closed operators}. Math. Z. {\bf 196}, no. 2, 189--201 (1987). 

\bibitem{Ha} M. Haase. {\em The functional calculus for sectorial operators}. Operator theory: Advances and applications {\bf169}, Birkh\"auser Verlag (2006).

\bibitem{LSU} O. A. Lady\v zenskaya, V. A. Solonnikov, N. N Ural'ceva. {\em Linear and quasi-linear equations of parabolic type}. Translations of Mathematical Monographs {\bf 23}, Amer. Math. Soc. Providence, RI, (1968).

\bibitem{LV} P. Lu, L. Ni, J. L. V\'azquez, C. Villani. {\em Local Aronson-B\'enilan estimates and entropy formulae for porous medium and fast diffusion equations on manifolds}. J. Math. Pures Appl. {\bf 91}, no. 1, 1--19 (2009).

\bibitem{Ma} R. Mazzeo, Y. Rubinstein, N Sesum. {\em Ricci flow on surfaces with conic singularities}. Anal. PDE {\bf 8}, no. 4, 839--882 (2015).

\bibitem{KW1} N. Kalton, L. Weis. {\em The $H^{\infty}$-calculus and sums of closed operators}. Math. Ann. {\bf 321}, no. 2, 319--345 (2001).

\bibitem{KuWe} P. C. Kunstmann, L. Weis. {\em Perturbation theorems for maximal $L_{p}$-regularity}. Ann. Scuola Norm. Sup. Pisa, Classe di Scienze, S\'erie 4, {\bf 30}, no. 2, 415--435 (2001).

\bibitem{Ro} N. Roidos. {\em Closedness and invertibility for the sum of two closed operators}. Adv. Oper. Theory {\bf 3}, no. 3, 582--605 (2018).

\bibitem{Ro2} N. Roidos. {\em Local geometry effect on the solutions of evolution equations: The case of the Swift-Hohenberg equation on closed manifolds}. arXiv:1612.08766.

\bibitem{RS2} N. Roidos, E. Schrohe. {\em Bounded imaginary powers of cone differential operators on higher order Mellin-Sobolev spaces and applications to the Cahn-Hilliard equation}. J. Differential Equations {\bf 257}, no. 3, 611--637 (2014).

\bibitem{RS3} N. Roidos, E. Schrohe. {\em Existence and maximal $L^{p}$-regularity of solutions for the porous medium equation on manifolds with conical singularities}. Comm. Partial Differential Equations {\bf 41}, no. 9, 1441--1471 (2016).

\bibitem{RS1} N. Roidos, E. Schrohe. {\em The Cahn-Hilliard equation and the Allen-Cahn equation on manifolds with conical singularities}. Comm. Partial Differential Equations {\bf 38}, no. 5, 925--943 (2013).

\bibitem{Sh} E. Schrohe, J. Seiler. {\em The resolvent of closed extensions of cone differential operators}. Can. J. Math. {\bf 57}, no. 4, 771--811 (2005).

\bibitem{Shao} Y. Shao. {\em Global solutions to the porous medium equations on singular manifolds}. arXiv:1606.01233.

\bibitem{Va} J. L. V\'azquez. {\em The porous medium equation, mathematical theory}. Oxford Mathematical Monographs. Oxford University Press (2007).


\end{thebibliography}
\end{document}